\documentclass[12pt]{amsart}
\usepackage{amsmath,amssymb,amsthm,amsrefs,bbm,esvect,float,graphicx,mathrsfs}
\usepackage{mathtools}
\usepackage{scalerel}
\usepackage{enumerate}
\usepackage{hyperref}
\makeatletter
\renewcommand{\BibLabel}{%
  \Hy@raisedlink{\hyper@anchorstart{cite.\CurrentBib}\relax\hyper@anchorend}%
  [\thebib]%
}
\makeatother
\usepackage{cite}
\usepackage{color}
\usepackage[usestackEOL]{stackengine}
\usepackage[a4paper,width=16.5cm,top=3cm,bottom=2.8cm]{geometry}

\newtheorem{theorem}{Theorem}[section]
\newtheorem{proposition}[theorem]{Proposition}

\newtheorem{lemma}[theorem]{Lemma}
\newtheorem{definition}[theorem]{Definition}
\newtheorem{remark}[theorem]{Remark}

\theoremstyle{plain}
\newtheorem*{theorem*}{Theorem}
\newtheorem*{question*}{Question}
\newtheorem*{example*}{Example}

\def\dashint{\,\ThisStyle{\ensurestackMath{%
			\stackinset{c}{.2\LMpt}{c}{.5\LMpt}{\SavedStyle-}{\SavedStyle\phantom{\int}}}%
		\setbox0=\hbox{$\SavedStyle\int\,$}\kern-\wd0}\int}
\newcommand{\rn}{\mathbb{R}^n}

\newcommand{\zn}{\mathbb{Z}^n}

\newcommand{\zed}{\mathbb{Z}}

\newcommand{\honern}{{H}^1(\mathbb{R}^n)}

\newcommand{\hprn}{{H}^p(\mathbb{R}^n)}

\newcommand{\dhprn}{{H}_d^p(\mathbb{R}^n)}
\newcommand{\hqrn}{{H}^q(\mathbb{R}^n)}
\newcommand{\dhqrn}{{H}_d^q(\mathbb{R}^n)}
\newcommand{\hrrn}{{H}^r(\mathbb{R}^n)}

\newcommand{\hprnn}[1]{\|#1\|_{\hprn}}

\newcommand{\hqrnn}[1]{\|#1\|_{\hqrn}}

\newcommand{\hrrnn}[1]{\|#1\|_{\hrrn}}

\newcommand{\bmorn}{{BMO}(\rn)}
\newcommand{\bmor}{{BMO}(\mathbb{R})}

\newcommand{\bmornn}[1]{\|#1\|_{{BMO}(\rn)}}
\newcommand{\bmoronen}[1]{\|#1\|_{{\bmor}}}

\newcommand{\bmodyadicrn}{{BMO_d}(\rn)}

\newcommand{\lonern}{{L}^1(\mathbb{R}^n)}

\newcommand{\lonernn}[1]{\|#1\|_{\lonern}}

\newcommand{\lprn}{{L}^p(\mathbb{R}^n)}
\newcommand{\lqrn}{{L}^q(\mathbb{R}^n)}
\newcommand{\lrrn}{{L}^r(\mathbb{R}^n)}

\newcommand{\lpn}[1]{\|#1\|_{{L}^p}}
\newcommand{\lprnn}[1]{\|#1\|_{\lprn}}
\newcommand{\lqrnn}[1]{\|#1\|_{\lqrn}}
\newcommand{\lrrnn}[1]{\|#1\|_{\lrrn}}

\newcommand{\linfty}{{L}^{\infty}}
\newcommand{\linftyrn}{{L}^{\infty}(\mathbb{R}^n)}

\newcommand{\linftyn}[1]{\|#1\|_{\linfty}}
\newcommand{\linftyrnn}[1]{\|#1\|_{\linftyrn}}

\newcommand{\lonernloc}{{L}_{loc}^1(\mathbb{R}^n)}

\newcommand{\avr}[2]{\left\langle#1\right\rangle_{#2}}

\newcommand{\hlip}[2]{\dot{\Lambda}^{#1}(\mathbb{R}^{#2})}

\newcommand{\hlipn}[3]{\|#1\|_{\hlip{#2}{#3}}}
\newcommand{\dhlipalpharn}[1]{\|#1\|_{\dot{\Lambda}_d^{\alpha}(\rn)}}

\begin{document}
	
	\title{The Operator Norm of Paraproducts on Hardy Spaces}
	\author{Shahaboddin Shaabani}
	\address{Department of Mathematics and Statistics\\Concordia University}
	\email{shahaboddin.shaabani@concordia.ca}
	
	\begin{abstract}
For a tempered distribution \( g \), and \( 0 < p, q, r < \infty \) with \(\frac{1}{q} = \frac{1}{p} + \frac{1}{r}\), we show that the operator norm of a Fourier paraproduct \(\Pi_g\), of the form
\[
\Pi_{g}(f) := \sum_{j \in \mathbb{Z}} (\varphi_{2^{-j}} * f) \Delta_jg,
\]
from \( H^p(\mathbb{R}^n) \) to \( \dot{H}^q(\mathbb{R}^n) \) is comparable to \( \|g\|_{\dot{H}^r(\mathbb{R}^n)} \). We also establish a similar result for dyadic paraproducts acting on dyadic Hardy spaces.

	\end{abstract}	
	\date{}
	\maketitle
	\section{Introduction}	
Let $\Delta_j$ ($j \in \zed$) be the Littlewood-Paley operators associated with a Schwarz function $\psi$, and let $S_j$ be its partial sum operators (precise definitions will be given in the next section). For an integer $l$ and tempered distributions $f$ and $g$, the bilinear form 
\begin{equation*}
	\Pi(f, g) := \sum_{j \in \zed} S_{j-l}(f) \Delta_j(g),
\end{equation*}
is an example of a paraproduct.
 There are also dyadic counterparts of these bilinear forms, which will be discussed later. Paraproducts are among the most fundamental bilinear forms in harmonic analysis and PDEs. Loosely speaking, they are "half products," and historically, their first appearance was in Bony's para-differential calculus \cite{MR0631751}, where they were used to extend the work of Coifman and Meyer on pseudo-differential calculus with minimal regularity assumptions on their symbols \cite{MR0518170}.
  Later, they were used in the proof of the $T(1)$ theorem by David and Journé \cite{MR0763911}, and since T. Figiel's work on the representation of singular integral operators in terms of simpler dyadic ones \cite{MR1110189}, they have proven to be an essential component of such operators \cite{MR2912709}. It is well-known that for real $p > 0$ and sufficiently large $l$, the operator $\Pi$ is bounded from $\hprn \times \bmorn$ to $\hprn$ \cite{MR1934198,MR1853518}, where $\hprn$ is the real Hardy space and $\bmorn$ is the space of functions of bounded mean oscillation \cite{MR0131498,MR0447953}.
  Also, for any triple of positive real numbers $(p, q, r)$ with $\frac{1}{p} + \frac{1}{r} = \frac{1}{q}$, the bilinear form $\Pi$ maps $\hprn \times \hrrn$ to $\hqrn$ \cite{MR2730492,MR1853518}. In addition, the boundedness properties of analogs of these objects in other settings, including multi-parameter and multi-linear settings, have been studied extensively. We refer the reader to \cite{MR2134868,MR2313844,MR2320408,MR3375865} and the references therein. The purpose of the present paper is to show that these bounds cannot be improved in the sense that by freezing $g$, the operator norm of the corresponding operator, $\Pi_g$, is comparable to a norm of $g$ predicted by the current bounds. See \ref{dyadictheorem} and \ref{continuousparp}.\\

Let us begin by reviewing the current known results in this area. The first result on the operator norm of paraproducts seems to be in \cite{MR2381883}, where it is shown that for the dyadic paraproduct operator $\pi_g$, we have
\[
\|\pi_g\|_{L^p(\mathbb{R}) \to L^p(\mathbb{R})} \simeq \|g\|_{\bmor}, \quad 1 < p < \infty.
\]
Also, recently the author in \cite{hänninen2023weighted} has shown among other things that for $\pi_g$ it holds that
\[
\|\pi_g\|_{\lprn\to\lqrn}\simeq \|g\|_{\dot{L}^r(\rn)}, \quad \frac{1}{q}=\frac{1}{p}+\frac{1}{r}, \quad 1<q<p<\infty.
\]
The result in \cite{hänninen2023weighted} is stronger than what we have stated here and was obtained in the Bloom setting, but we are not concerned with that here. To the best of our knowledge, these are all the results known about the operator norm of paraproducts. The ideas used in \cite{hänninen2023weighted} are similar to those methods employed in \cite{MR4338459}, where the author completed the characterization of the operator norm of commutators of a non-degenerate Calderón-Zygmund operator $T$ and pointwise multiplication with a locally integrable function $b$.
 In fact, it is shown in \cite{MR4338459} that for $1<p,q<\infty$ we have:
\[
	\|\left[T,b\right]\|_{\lprn\to\lqrn}\simeq \begin{cases*}
		\|b\|_{\dot{L}^{r}(\rn)}&$\frac{1}{q}=\frac{1}{p}+\frac{1}{r}$,\quad $p>q$\\
		\bmornn{b}& \quad \quad \quad \quad \quad $p=q$\\
		\|b\|_{\dot{C}^{\alpha}(\rn)}& $\frac{1}{q}=\frac{1}{p}-\frac{\alpha}{n}$, \quad
		 $p<q\le \frac{np}{(n-p)_+}$,
	\end{cases*}
\]
where in the above $\dot{L}^{r}(\rn)$ denotes $\lrrn$ modulo constants and $\dot{C}^{\alpha}(\rn)$ is the homogeneous Holder space. The strategy of the proof, used both in \cite{MR4338459} and \cite{hänninen2023weighted}, is to bound the local mean oscillation of the symbol of the operator by testing it with suitable test functions. Then, using boundedness of the operator will finish the job for all cases except for $p>q$. To resolve this, the local mean oscillation inequality \cite{MR2721744,MR4007575,operatorfree} is employed to obtain an appropriate sparse domination of the symbol. Then, after using a duality argument and a probabilistic linearization, the result follows from boundedness of the operator.\\
	
Here, although we are dealing with the operator from $H^p$ to $H^q$, for all cases where $p \le q$, our approach is essentially the same as in \cite{MR2381883}, \cite{MR4338459}, and \cite{hänninen2023weighted}. To show that the symbol belongs to $\bmorn$ or $\hlip{\alpha}{n}$, we bound an appropriate mean oscillation of the operator's symbol.
 However, when $q < p$, and especially when $q < 1$, an argument based on duality will not work. This is because the Hahn-Banach theorem fails for $H^q$ when $0 < q < 1$ \cite{MR0259579}, so we cannot guarantee that the operator norm of the operator and its adjoint are the same. Instead, by using a suitable sparse domination for the square function of the symbol $g$, we can handle all cases where $0 < q < p < \infty$.
 This sparse domination is not new; it is a special case of the general method formulated in \cite{operatorfree}. This method shows that almost all known sparse domination results follow a general approach, which we use here. We should also mention that we first encountered this idea in \cite{MR2157745}, where it is used to construct an atomic decomposition for dyadic $H^1$, based on the square function. Notably, the construction idea in \cite{MR2157745} is derived from \cite{MR0671315}, where it is shown that functions in Hardy spaces have an atomic decomposition into atoms, whose supporting cubes form a sparse family.

\section{preliminaries}
\subsection{Notation}
Throughout the paper, we use $\lesssim$, $\gtrsim$, and $\simeq$ to suppress constants and parameters in inequalities that are not crucial to our discussion; this will be clear from the context. We use $B_r(x)$ to denote the ball of radius $r$ centered at $x$. A cube $Q$ in $\rn$ refers to a cube with sides parallel to the coordinate axes. We denote its Lebesgue measure by $|Q|$, its side length by $l(Q)$, and its dilation by a factor $a$ as $aQ$. A dyadic cube is a cube $Q$ of the form $Q = 2^k(m + [0,1]^n)$, where $k \in \zed$ and $m \in \zn$. For such a cube, all $2^n$ cubes obtained by bisecting its sides are also dyadic and are called its children. Any dyadic cube $Q$ is a child of a unique cube, called the parent of $Q$, and denoted by $\hat{Q}$. We use $\mathcal{D}(Q)$ to denote the family of all dyadic cubes within $Q$, and $\mathcal{D}$ for all dyadic cubes in $\rn$.
For a locally integrable function $f$, its non-increasing rearrangement is denoted by
\[
f^*(t) := \inf\left\{ s \mid |\{ |f| > s \} | \le t \right\}, \quad t > 0,
\]
and its average with respect to the Lebesgue measure over a cube $Q$ is given by
\[
\avr{f}{Q} := \frac{1}{|Q|} \int_{Q} f.
\]
Additionally, for $0 < p < \infty$, we use $\text{osc}_p(f, Q)$ to denote its $p$-mean oscillation and $\text{osc}(f, Q)$ to denote its pointwise oscillation over $Q$:
\[
\text{osc}_p(f, Q) := \avr{|f - \avr{f}{Q}|^p}{Q}^{\frac{1}{p}}, \quad \text{osc}(f, Q) := \sup_{x, y \in Q} |f(x) - f(y)|.
\]

\subsection{Sparse Families}
Here, we review the definition of sparse families, one of their useful properties, and the general method of sparse domination introduced in \cite{operatorfree}.

\begin{definition}
	Let $0 < \eta < 1$, and let $\mathcal{C}$ be a family of cubes that are not necessarily dyadic. We say that $\mathcal{C}$ is \emph{$\eta$-sparse} if, for each $Q \in \mathcal{C}$, there exists a subset $E_Q \subseteq Q$ such that $|E_Q| \ge \eta |Q|$, and for any two cubes $Q, Q' \in \mathcal{C}$, the sets $E_Q$ and $E_{Q'}$ are disjoint.
\end{definition}

A useful property of sparse families, which is crucial for us here, is that although they may have many overlaps, they behave as if they are disjoint. A simple example of this phenomenon is given by the following lemma, which is well-known and whose proof can be found in \cite{MR4338459, hänninen2023weighted}.

\begin{lemma}\label{lemmasparse}
	For an $\eta$-sparse family of cubes $\mathcal{C}$, nonnegative numbers $\{a_Q\}_{Q \in \mathcal{D}}$, and $0 < p < \infty$, we have
	\begin{align}
		\eta^{\frac{1}{p}} \left( \sum_{Q \in \mathcal{C}} a_Q^p |Q| \right)^{\frac{1}{p}} 
		&\lesssim \lpn{\sum_{Q \in \mathcal{C}} a_Q \chi_Q} 
		\lesssim \eta^{-1} \left( \sum_{Q \in \mathcal{C}} a_Q^p |Q| \right)^{\frac{1}{p}}. \label{lpnormsparse}
	\end{align}
\end{lemma}

We now describe the general sparse domination method introduced in \cite{operatorfree}.

Let $\{f_Q\}_{Q \in \mathcal{D}}$ be a family of measurable functions defined on $\rn$, which we consider as ``localizations'' of a function or operator. For every two dyadic cubes $P$ and $Q$ with $P \subseteq Q$, let $f_{P,Q}$ be a measurable function satisfying
\[
|f_{P,Q}| \le |f_P| + |f_Q|,
\]
where $f_{P,Q}$ will serve as the ``difference'' between two localizations. Finally, for such a family of functions, the maximal sharp function is defined as
\begin{equation}
	m_Q^{\#} f(x) := \sup_{\substack{x \in P \\ P \subseteq Q}} \text{osc}(f_{P,Q}, P), \quad x \in Q.
\end{equation}
Then, it is proved in \cite{operatorfree} that the following holds:

\begin{theorem}[Lerner, Lorist, Ombrosi]\label{operatorfree}
	Let $\{f_Q\}_{Q \in \mathcal{D}}$ and $\{f_{P,Q}\}$ be as described above. For a dyadic cube $Q_0$ and $0 < \eta < 1$, there exists an $\eta$-sparse family of cubes $\mathcal{C}$ contained in $Q_0$ such that for almost every $x$ in $Q_0$,
	\begin{equation}
		f_{Q_0}(x) \lesssim \sum_{Q \in \mathcal{C}} \gamma_Q \chi_Q(x),
	\end{equation}
	where
	\begin{equation}
		\gamma_Q = (f_Q \chi_Q)^* \left(\frac{1 - \eta}{2^{n+2}} |Q| \right) + (m_Q^{\#} f)^* \left(\frac{1 - \eta}{2^{n+2}} |Q| \right).
	\end{equation}
\end{theorem}

As shown in \cite{operatorfree}, almost all known sparse dominations are special cases of the general theorem above. We refer the reader to \cite{operatorfree, MR4007575} and the references therein for a general theory of sparse domination and dyadic calculus.

\subsection{Dyadic Hardy Spaces}
Next, we review the definitions and basic properties of the dyadic Hardy spaces that we are concerned with, beginning with the Haar basis.

For a dyadic interval $I$, let $h_I = |I|^{-\frac{1}{2}} (\chi_{I^-} - \chi_{I^+})$, where $I^-$ and $I^+$ are the left and right halves of $I$, respectively. These functions form the well-known Haar basis for $L^2(\mathbb{R})$. In higher dimensions, the Haar basis is defined as follows: First, we let $h^1_I = h_I$ and $h^0_I = |I|^{-\frac{1}{2}} \chi_I$. Then, for a dyadic cube $Q = \prod_{j=1}^{n} I_j$ and a multi-index $i = (i_1, \ldots, i_n) \neq 0$, where each $i_j \in \{0,1\}$, we define
\[
h^i_Q = \prod_{j=1}^{n} h^{i_j}_{I_j}.
\]
For the Haar basis, the associated square function is defined as
\begin{equation}\label{defsquarfunc}
	S_d(f)(x) := \left( \sum_{Q \in \mathcal{D}} \sum_{i \neq 0} \left| \langle f, h^i_{Q} \rangle \right|^2 \frac{\chi_Q(x)}{|Q|} \right)^{\frac{1}{2}}, \quad x \in \rn,
\end{equation}
where $\langle \cdot, \cdot \rangle$ denotes the usual inner product in $L^2(\rn)$. Another important dyadic operator is the dyadic maximal operator, defined as
\begin{equation}\label{defdyadmax}
	M_d(f)(x) := \sup_{x \in Q} \left| \langle f \rangle_Q \right|,
\end{equation}
where the supremum is taken over all dyadic cubes $Q$ containing $x$.
\begin{definition}
	For $0 < p < \infty$, the dyadic Hardy space $H_d^p(\rn)$ is the completion of the space of locally integrable functions $f$ such that 
	\begin{equation*}
		\|f\|_{H_d^p(\rn)} := \|M_d(f)\|_{L^p(\rn)} < \infty.
	\end{equation*}
\end{definition}
When $1 \leq p < \infty$, the dyadic Hardy space is a Banach space with the above quantity as its norm. For $0 < p < 1$, however, this quantity is a quasi-norm, making $H_d^p(\mathbb{R}^n)$ a quasi-Banach space. It is also well known that for $1 < p < \infty$, $H_d^p(\mathbb{R}^n)$ is identical to $L^p(\mathbb{R}^n)$. There is also a closely related space that is not identical to $H_d^p(\mathbb{R}^n)$, even though it is denoted by the same notation in the literature. To define it properly, we make the following definition:

\begin{definition}
	A dyadic distribution $f$ is a family of complex numbers $\{f^i_Q : Q \in \mathcal{D}, i \ne 0\}$ and is formally written as
	\[
	f = \sum_{Q \in \mathcal{D}} \sum_{i \ne 0} \langle f, h^i_{Q} \rangle h^i_Q, \quad \langle f, h^i_{Q} \rangle := f^i_Q.
	\]
\end{definition}

\begin{definition}
	For $0 < p < \infty$, the dyadic Hardy space $\dot{H}_d^p(\rn)$ is the space of all dyadic distributions $f$ such that 
	\begin{equation*}
		\|f\|_{\dot{H}_d^p(\rn)} := \|S_d(f)\|_{L^p(\rn)} < \infty.
	\end{equation*}
\end{definition}
The important fact here is that $\dot{H}_d^p(\mathbb{R}^n)$ is the same as $H_d^p(\mathbb{R}^n)$ modulo constants, meaning that for $f\in\lonernloc$ we have
\begin{equation}\label{maxchardyadhp}
	\|f\|_{\dot{H}_d^p(\mathbb{R}^n)} \simeq_p \inf_{c \in \mathbb{C}} \|f - c\|_{H_d^p(\mathbb{R}^n)}.
\end{equation}
Now, we turn to the definition of the dual of Hardy spaces.

\begin{definition}
	For $0 \le \alpha < \infty$, the dyadic homogeneous Lipschitz space $\dot{\Lambda}_d^{\alpha}(\mathbb{R}^n)$ is the space of all locally integrable functions $f$, modulo constants, such that
	\begin{equation*}\label{defdyadlips}
		\|f\|_{\dot{\Lambda}_d^{\alpha}(\mathbb{R}^n)} := \sup_{Q \in \mathcal{D}} l(Q)^{-\alpha} \text{osc}_1(f, Q) < \infty.
	\end{equation*}
\end{definition}

In the above, $\dot{\Lambda}_d^{0}(\mathbb{R}^n)$ is identical with the dyadic $BMO$, denoted by $BMO_d(\mathbb{R}^n)$, or the space of functions with bounded mean oscillation on dyadic cubes. The first crucial fact about the above definition is that if, for a positive number $p$, we replace $\text{osc}_1(f, Q)$ with $\text{osc}_p(f, Q)$, we obtain nothing but the same space. More precisely, we have
\begin{equation}\label{johnnirenberg}
	\|f\|_{\dot{\Lambda}_d^{\alpha}(\mathbb{R}^n)} \simeq \sup_{Q \in \mathcal{D}} l(Q)^{-\alpha} \text{osc}_p(f, Q) < \infty, \quad \alpha \ge 0, \quad p > 0.
\end{equation}
Another important fact that we need later is that for $0 < p \le 1$, the space ${\dot{\Lambda}}_d^{\alpha}(\mathbb{R}^n)$ is the dual of $H_d^p(\mathbb{R}^n)$. Specifically,
\begin{equation}\label{hpdyadduality}
	{H_d^p(\mathbb{R}^n)}^* \cong \dot{\Lambda}_d^{n(\frac{1}{p} - 1)}(\mathbb{R}^n), \quad 0<p\le 1.
\end{equation}

See \cite{MR0372987,MR0440695,MR0448538} for the proof of these.
\subsection{Dyadic Paraproducts}
Now, we define and review the boundedness properties of dyadic paraproducts, which are the focus here.

\begin{definition}
	For a dyadic distribution $g$, the dyadic paraproduct operator with symbol $g$ is defined as
	\begin{equation}\label{defdyadpara}
		\pi_g(f) := \sum_{Q \in \mathcal{D}} \sum_{i \ne 0} \langle f \rangle_Q \langle g, h^i_{Q} \rangle h^i_Q, \quad f \in L^1_{\text{loc}}(\rn).
	\end{equation}
\end{definition}

In the above, $\pi_g(f)$ is a dyadic distribution, and when $f$ is a linear combination of finitely many Haar functions, it is a well-defined function. The next theorem contains boundedness properties of $\pi_g$ on dyadic Hardy spaces.

\begin{theorem*}[A]\label{theorema}
	For a dyadic distribution $g$ and real numbers $0 < p, q, r < \infty$, the following inequalities hold:
	\begin{align}
		&\|\pi_g(f)\|_{\dot{H}_d^q(\rn)} \lesssim \|g\|_{\dot{H}_d^r(\rn)} \|f\|_{H_d^p(\rn)}, \quad \frac{1}{q} = \frac{1}{p} + \frac{1}{r} \label{dyadhphq}\\
		&\|\pi_g(f)\|_{\dot{H}_d^{p^*}(\rn)} \lesssim \|g\|_{\dot{\Lambda}_d^{\alpha}(\rn)} \|f\|_{H_d^p(\rn)}, \quad \frac{1}{p^*} = \frac{1}{p} - \frac{\alpha}{n}, \quad 0 < \alpha p < n \label{dyadhpstarhp}\\
		&\|\pi_g(f)\|_{\dot{H}_d^{p}(\rn)} \lesssim \|g\|_{BMO_d(\rn)} \|f\|_{H_d^p(\rn)} \label{dyadhphp}
	\end{align}
\end{theorem*}

Let us briefly discuss the reasons behind these inequalities. The first inequality \eqref{dyadhphq} follows from the pointwise bound
\[
S_d(\pi_g(f))(x) \le M_d(f)(x) S_d(g)(x),
\]
the Hölder inequality, and the maximal characterization of dyadic Hardy spaces \eqref{maxchardyadhp}. For inequality \eqref{dyadhpstarhp}, we have an analog of the classical Hedberg inequality, which we could not find in the literature, so we decided to prove it here.

\begin{proposition}
	For $0<\alpha n<p<\infty$, a dyadic distribution $g$, and a locally integrable function $f$, it holds
	\begin{equation}\label{dyadhedberg}
		S_d(\pi_g(f))(x)\lesssim\|g\|_{\dot{\Lambda}_d^{\alpha}(\rn)}\|f\|_{H_d^p(\rn)}^{\frac{\alpha p}{n}}M_d(f)(x)^{\frac{p}{p^*}}.
	\end{equation}
\end{proposition}

\begin{proof}
	After normalization, we may assume $\|g\|_{\dot{\Lambda}_d^{\alpha}(\rn)} = \|f\|_{H_d^p(\rn)} = 1$. Let $Q$ be a dyadic cube. The cancellation property of Haar functions ($\int h_Q^i = 0$), and the triangle inequality imply
	\begin{equation*}
		|\avr{g,h_Q^i}{}| \le \avr{|g - \avr{g}{Q}|, |h_Q^i|}{} \le \text{osc}_1(g, Q) |Q|^{\frac{1}{2}} \le |Q|^{\frac{\alpha}{n} + \frac{1}{2}}.
	\end{equation*}
	For $\avr{f}{Q}$, we have the following two competing bounds:
	\begin{equation}\label{thefirstboundherd}
		|\avr{f}{Q}| \le \inf_{x \in Q} M_d(f)(x), \quad |\avr{f}{Q}| \lesssim |Q|^{-1} \|f\|_{H_d^p(\rn)} \|\chi_Q\|_{{H_d^p(\mathbb{R}^n)}^*},
	\end{equation}
where the right-hand side inequality follows from duality.
When $1<p<\infty$, it is clear that $\|\chi_Q\|_{{H_d^p(\mathbb{R}^n)}^*}\le |Q|^{1-\frac{1}{p}}$, and the duality relation \eqref{hpdyadduality} gives a similar bound holds for $0<p\le 1$. To see this, note that if $R \subseteq Q$ is a dyadic cube, then $\text{osc}_1(\chi_Q, R) = 0$. For $Q \subsetneq R$, we have
 \[
 \text{osc}_1(\chi_Q, R) \le 2 \avr{\chi_Q}{R} \le 2 |R|^{-1} |Q| \le 2 |R|^{\frac{1}{p}-1} |Q|^{1 - \frac{1}{p}},
 \]
 which implies 
 \begin{equation}\label{hedbergseconddyad}
 	\|\chi_Q\|_{\dot{\Lambda}_d^{n(\frac{1}{p}-1)}(\rn)} \le 2 |Q|^{1 - \frac{1}{p}}.
 \end{equation}
 Combining \eqref{thefirstboundherd} and \eqref{hedbergseconddyad}, we get
 \begin{equation*}
 	|\avr{f}{Q}| \lesssim \min\left\{\inf_{x \in Q} M_d(f)(x), |Q|^{-\frac{1}{p}}\right\}, \quad 0<p<\infty.
 \end{equation*}
 Finally, we estimate the square function as
 \begin{align*}
 	S_d(\pi_g(f))(x) &\le \sum_{Q \in \mathcal{D}} \sum_{i \neq 0} |\avr{f}{Q}| |\avr{g, h_Q^i}{}| |Q|^{-\frac{1}{2}} \chi_Q(x) \\
 	&\lesssim \sum_{Q \in \mathcal{D}} \min\left\{M_d(f)(x), |Q|^{-\frac{1}{p}}\right\} |Q|^{\frac{\alpha}{n}} \chi_Q(x) \\
 	&\le M_d(f)(x) \sum_{M_d(f)(x) \le |Q|^{-\frac{1}{p}}} |Q|^{\frac{\alpha}{n}} \chi_Q(x) + \sum_{M_d(f)(x) > |Q|^{-\frac{1}{p}}} |Q|^{\frac{\alpha}{n} - \frac{1}{p}} \chi_Q(x) \\
 	&\lesssim M_d(f)(x)^{\frac{p}{p^*}},
 \end{align*}
 which proves the claim.

\end{proof}

Now, taking the $L^{p^*}$ norm of \eqref{dyadhedberg} proves \eqref{dyadhpstarhp}. The last inequality \eqref{dyadhphp} is the most complex one, so we only outline the main steps to prove it. The $L^2$ boundedness of $\pi_g$ follows from the Haar characterization of $\bmodyadicrn$ in terms of Carleson measures and the Carleson embedding theorem. Boundedness on $L^p$ comes from Calderón-Zygmund theory, and boundedness on $H_d^p$ is derived from the atomic decomposition. We refer the reader to \cite{MR2157745} for more on dyadic Hardy spaces and a proof of this. See also \cite{MR2381883} for another proof of \eqref{dyadhphp}. We will now shift from the dyadic setting to the Fourier setting and provide the necessary definitions.

\subsection{Littlewood-Paley Operators}
We begin by fixing some notation. For a Schwartz function $f$, the Fourier transform is defined by
\[
\hat{f}(\xi):=\int_{\rn}f(x)e^{-2\pi i\left\langle x,\xi\right\rangle}dx,
\]
and the convolution of two functions is defined as
\[
f*g(x):=\int_{\rn}f(x-y)g(y)dy.
\]
Moreover, we use the following notation for translations and dilations of functions:
\[
(\tau^{x_0}f)(x):=f(x-x_0), \quad \delta^t(f)(x):=f(t^{-1}x), \quad f_t(x):=t^{-n}\delta^t(f)(x), \quad t>0, \quad x_0, x\in\rn,
\]
and below we summarize some of their useful properties:
\[
\tau^{x_0}(f*g)=(\tau^{x_0}f)*g=f*(\tau^{x_0}g), \quad f_t*\delta^s g=\delta^s(f_{s^{-1}t}g).
\]
Now, let $\psi$ be a Schwartz function with its Fourier transform supported in an annulus away from the origin and infinity, meaning
\begin{equation}\label{psiconditins}
	\text{supp}(\hat{\psi}) \subseteq \left\{\textbf{\textit{a}} \le |\xi| \le \textbf{\textit{b}}\right\}, \quad 0 < \textbf{\textit{a}} < \textbf{\textit{b}} < \infty,
\end{equation}
and such that
\begin{equation}\label{partionofunity}
	\sum_{j \in \zed} \hat{\psi}(2^{-j} \xi) = 1, \quad \xi \neq 0,
\end{equation}
(see \cite{MR2445437} for examples of such functions). Then, the Littlewood-Paley operators associated with $\psi$ are defined as
\[
\Delta_j^\psi(f) := \psi_{2^{-j}} * f, \quad j \in \zed,
\]
and the partial sum operators are defined as
\[
S_j^\psi(f) := \sum_{k \le j} \Delta_k^\psi(f), \quad S_j(f) = \Psi_{2^{-j}} * f,
\]
where
\[
\hat{\Psi} :=
\begin{cases*}
	\sum_{k \le 0} \hat{\psi}(2^k \xi) & $\xi \neq 0$ \\
	1 & $\xi = 0$
\end{cases*}.
\]
In the above, $\Psi$ is a Schwartz function whose Fourier transform is supported in a ball around the origin and equals to $1$ in a smaller neighborhood (throughout the paper, we use capital Greek letters for the kernel functions of the partial sum operators). Similarly to the dyadic case, the square function with respect to $\psi$ is defined as
\[
S_\psi(f)(x) := \left(\sum_{j \in \zed} |\Delta_j^\psi(f)(x)|^2\right)^{\frac{1}{2}}.
\]
For simplicity, we drop the dependence on $\psi$ when it is clear from the context. An important feature of the Littlewood-Paley pieces $\Delta_j(f)$ is that their Fourier transform is localized at the scale $2^j$, which means they behave like a constant at scale $2^{-j}$. This feature can be expressed through Plancherel-Polya-Nikolskij type inequalities.

\begin{theorem}\label{polyainequality}
	Let $f$ be a tempered distribution whose Fourier transform is supported in a ball of radius $t > 0$. Then
	\begin{itemize}
		\item For $0 < p \le q \le \infty$, we have
		\begin{equation}\label{polyainequalitylplq}
			\lqrnn{f} \lesssim t^{n\left(\frac{1}{p} - \frac{1}{q}\right)} \lprnn{f}.
		\end{equation}
		\item There exists a constant $c$ such that for $0 < h \le c t^{-1}$, and any sequence of numbers $\left\{x_k: x_k \in h([0,1]^n + k), \; k \in \zn \right\}$, we have
		\begin{equation}\label{polyalittlelp}
			\|\{f(x_k)\}_{k \in \zn}\|_{l^p(\zn)} \simeq h^{-\frac{n}{p}} \lprnn{f}, \quad 0 < p \le \infty.
		\end{equation}
	\end{itemize}	

\end{theorem}
See \cite{MR0781540} for the proof.
\subsection{Real Hardy Spaces} Now, we recall the definition of the real Hardy spaces.\\
For a Schwartz function $\varphi$, the associated maximal operator $M_\varphi$ and the non-tangential maximal operator $M_{\varphi}^*$ are defined as
\[
M_{\varphi}(f)(x) := \sup_{t > 0} |\varphi_t * f(x)|, \quad M_{\varphi}^*(f)(x) := \sup_{t > 0} \sup_{|x - y| \le t} |\varphi_t * f(y)|.
\]

\begin{definition}
	For a Schwartz function $\varphi$ with $\int \varphi = 1$ and $0 < p < \infty$, the Hardy space $H^p(\rn)$ is the space of all tempered distributions $f$ such that
	\[
	\hprnn{f} := \lprnn{M_{\varphi}(f)} < \infty.
	\]
\end{definition}
The above quantity defines a quasi-norm when $0 < p < 1$ and a norm when $1 \le p < \infty$, which makes $\hprn$ into a (quasi) Banach space. The space $\hprn$ is independent of $\varphi$, and for any other choice of $\varphi$, the corresponding (quasi) norms are comparable to each other. Also, for any Schwartz function $\tilde{\varphi}$, we have
\[
\lprnn{M_{\tilde{\varphi}}^*(f)} \lesssim \hprnn{f}, \quad 0 < p < \infty,
\]
and if $\int \tilde{\varphi} = 1$, the above inequality becomes an equivalence, with bounds depending only on $p$, $n$, and finitely many Schwartz semi-norms of $\varphi$ and $\tilde{\varphi}$. Furthermore, with a minor difference from the dyadic case, the square function characterization holds as well. To be more precise, let $\dot{H}^p(\rn)$ be the Hardy space $\hprn$ modulo polynomials, or equivalently, the space of all tempered distributions with
\[
\|f\|_{\dot{H}^p(\rn)} := \inf_{P \in \mathbb{P}_n} \hprnn{f - P} < \infty,
\]
where $\mathbb{P}_n$ is the space of all polynomials on $\rn$. Then, for any choice of Littlewood-Paley operators $\{\Delta_j\}_{j \in \zed}$, we have
\[
\lprnn{S(f)} \simeq \|f\|_{\dot{H}^p(\rn)}, \quad 0 < p < \infty.
\]
Finally, just as in the dyadic case, for $1 < p < \infty$, $\hprn$ coincides with ${L}^p(\rn)$ \cite{MR0447953, MR1232192, MR2463316}. Next, we recall the duals of Hardy spaces.

\begin{definition}
	For $0 < \alpha < \infty$, the homogeneous Lipschitz space $\hlip{\alpha}{n}$ is the Banach space of all functions $f$ with
	\[
	\hlipn{f}{\alpha}{n} := \sup_{x \in \rn} \sup_{h \ne 0} |h|^{-\alpha} |D_h^{\left[\alpha\right]+1}(f)(x)| < \infty,
	\]
	where $D_h$ is the forward difference operator defined as $D_h(f)(x) = f(x + h) - f(x)$, and $[\alpha]$ denotes the largest integer not greater than $\alpha$.
\end{definition}

Modulo polynomials, the homogeneous Lipschitz space $\hlip{\alpha}{n}$ can be characterized in terms of Littlewood-Paley pieces. Specifically, let $\ddot{\Lambda}^\alpha(\rn)$ be the space of all tempered distributions $f$ with 
\[
\|f\|_{\ddot{\Lambda}^\alpha(\rn)} := \inf_{P \in \mathbb{P}_n} \hlipn{f - P}{\alpha}{n} < \infty.
\]
Then, for any choice of Littlewood-Paley operators, we have
\[
\|f\|_{\ddot{\Lambda}^\alpha(\rn)} \simeq \sup_{j \in \zed} 2^{\alpha j} \linftyrnn{\Delta_j(f)}.
\]
Similarly to the dyadic case, we have
\[
\hprn^* \cong \hlip{n\left(\frac{1}{p} - 1\right)}{n}, \quad 0 < p < 1,  \quad \honern^* \cong \bmorn,
\]
where $\bmorn$ is the space of functions of bounded mean oscillation equipped with the norm
\[
\bmornn{f} := \sup_{Q} \text{osc}_1(f, Q),
\]
where the $\sup$ is taken over all cubes in $\rn$. It is well-known that for any positive $p$, if we replace $\text{osc}_1$ with $\text{osc}_p$, we get an equivalent norm \cite{MR0807149, MR1232192, MR2463316}. It is also useful to define $\dot{BMO}(\rn)$, or $\bmorn$ modulo polynomials, with the usual definition of the norm
\[
\|f\|_{\dot{BMO}(\rn)} := \inf_{P \in \mathbb{P}_n} \bmornn{f - P}.
\]

\subsection{Fourier Paraproducts}
Finally, we discuss the type of paraproducts mentioned in the introduction.

\begin{definition}
	Let $\psi$ be a Schwartz function satisfying \eqref{psiconditins}. For a tempered distribution $g$ and a Schwartz function $\varphi$, the paraproduct operator $\Pi_{g, \varphi}$ with symbol $g$ is formally defined as
	\[
	\Pi_{g, \varphi}(f) := \sum_{j \in \zed} (\varphi_{2^{-j}} * f)  \Delta_j^\psi(g).
	\]
\end{definition}

The meaning of convergence in the above sum is not clear unless we impose some restrictions on $f$ and $\varphi$. One situation where the above operator is well-defined is when the support of the Fourier transform of $\varphi$ lies in a compact subset of $\rn$. In this case, for a Schwartz function $f$ with Fourier transform supported in an annulus, the sum is finite and hence yields a well-defined smooth function. Furthermore, when the support of $\hat{\varphi}$ lies in a ball strictly within the annulus containing the support of $\hat{\psi}$, the boundedness of this operator on various Hardy spaces is well-known and we present it here as a theorem \cite{MR1853518, MR2463316, MR1232192}.

 \begin{theorem*}[B]\label{theoremb}
 	Suppose $\varphi$ is a Schwartz function whose Fourier transform is supported in a ball around the origin with radius $\textbf{a}' < \textbf{a}$, where $\textbf{a}$ is as in \eqref{psiconditins}. Then, for a distribution $g$ and real numbers $0 < p, q, r < \infty$, the following inequalities hold:
 	\begin{align}
 		&\|\Pi_{g, \varphi}(f)\|_{\dot{H}^q(\rn)} \lesssim \|g\|_{\dot{H}^r(\rn)} \|f\|_{H^p(\rn)}, \quad \frac{1}{q} = \frac{1}{p} + \frac{1}{r}, \label{hphq} \\
 		&\|\Pi_{g, \varphi}(f)\|_{\dot{H}^{p^*}(\rn)} \lesssim \|g\|_{\ddot{\Lambda}^\alpha(\rn)} \|f\|_{H^p(\rn)}, \quad \frac{1}{p^*} = \frac{1}{p} - \frac{\alpha}{n}, \quad 0 < \alpha p < n, \label{hpstarhp} \\
 		&\|\Pi_{g, \varphi}(f)\|_{\dot{H}^p(\rn)} \lesssim \|g\|_{\dot{BMO}(\rn)} \|f\|_{H^p(\rn)}. \label{hphp}
 	\end{align}
 \end{theorem*}
 
 The reasons for the validity of these inequalities are similar to those in the dyadic case, although there is a minor difference. The source of this difference is that in the definition of $\Pi_{g, \varphi}$, there is always some overlap between the Fourier supports of consecutive terms. Consequently, we cannot guarantee that a term like $(\varphi_{2^{-j}}*f) \Delta_j(g)$ is a Littlewood-Paley piece of $\Pi_{g, \varphi}$. However, since the Fourier support of the product of two functions is contained within the algebraic sum of their Fourier supports, for $\textbf{\textit{a}}' < \textbf{\textit{a}}$, the Fourier support of $(\varphi_{2^{-j}}*f) \Delta_j(g)$ remains away from the origin and around the annulus where the Fourier transform of $\Delta_j(g)$ is supported. Therefore, for a sufficiently large natural number $m$ depending only on $\textbf{\textit{a}}'$, $\textbf{\textit{a}}$, and $\textbf{\textit{b}}$, the Fourier supports of the terms in
 \begin{equation}\label{pigdefinition}
 	\Pi_{i, g, \varphi}(f) := \sum_{j \in m\zed + i} (\varphi_{2^{-j}}*f) \Delta_j(g), \quad 0 \le i < m
 \end{equation}
 are all sufficiently far from each other. Thus, by choosing an appropriate Littlewood-Paley operator $\{\Delta^\theta_j\}_{j \in \zed}$ such that $\hat{\theta}$ equals $1$ in a neighborhood of the support of $\hat{\psi}$, we have
 \[
 \Delta^\theta_k(\varphi_{2^{-j}}*f \Delta_j(g)) = \delta_{k, j} \varphi_{2^{-j}}*f \Delta_j(g), \quad k \in \zed, \quad j \in m\zed + i.
 \]
 
 This implies that 
 \[
 S_\theta(\Pi_{i, g, \varphi})(f) = \left(\sum_{j \in m\zed + i} |\varphi_{2^{-j}}*f \Delta_j(g)|^2 \right)^{\frac{1}{2}}, \quad 0 \le i < m.
 \]
 Arguments similar to those used in the dyadic case can be applied to the operators $\Pi_{i, g, \varphi}$. Since 
 \[
 \Pi_{g, \varphi} = \sum_{0 \le i < m} \Pi_{i, g, \varphi},
 \]
 we conclude that the same results hold for $\Pi_{g, \varphi}$ \cite{MR2463316}. Having established our notation, provided the necessary definitions, and recalled the essential facts, we now proceed to the next section of this article.

\section{The operator norm of Dyadic paraproducts on dyadic hardy spaces}

Our main results in this section are as follows:
\begin{theorem}\label{dyadictheorem}
	Let $g$ be a dyadic distribution. Then we have 
	\begin{align}
		&\tag{i}\|\pi_g\|_{H_d^p(\rn) \to \dot{H}_d^q(\rn)} \simeq \|g\|_{\dot{H}_d^r(\rn)}, \quad \frac{1}{q} = \frac{1}{p} + \frac{1}{r}, \quad 0 < p, q, r < \infty, \label{thefirstcasdaydic} \\
		&\tag{ii}\|\pi_g\|_{H_d^p(\rn) \to \dot{H}_d^{p^*}(\rn)} \simeq \dhlipalpharn{g}, \quad \frac{1}{p^*} = \frac{1}{p} - \frac{\alpha}{n}, \quad 0 \le \alpha p < n, \quad 0 < p < \infty. \label{thesecondcase dyadic}
	\end{align}
\end{theorem}
To prove \eqref{thefirstcasdaydic}, we need the following rather general theorem.

\begin{theorem}\label{mainthereom}
	Let $\{g_Q\}_{Q \in \mathcal{D}}$ be a sequence of nonnegative numbers indexed by dyadic cubes, where $0 < q, r, s, p < \infty$ and $\frac{1}{q} = \frac{1}{p} + \frac{1}{r}$. Suppose there exists a constant $A$ such that for all step functions $f$ with compact support, the following inequality holds:
	\begin{equation}\label{assumption}
		\left\|\sum_{Q \in \mathcal{D}} |\langle f \rangle_{Q}|^s g_Q \chi_Q \right\|_{L^q(\rn)} \le A \|f\|_{H_d^{sp}(\rn)}^s.
	\end{equation}
	Then we have
	\begin{equation}\label{conclusion ofthe maintheorem}
		\left\|\sum_{Q \in \mathcal{D}} g_Q \chi_Q \right\|_{L^r(\rn)} \lesssim A.
	\end{equation}
\end{theorem}

Let us accept this and prove Theorem \ref{dyadictheorem}.
\begin{proof}[Proof of Theorem \ref{dyadictheorem}]
	The upper bounds for the operator norm of $\pi_g$ are covered by Theorem (A), so we need to prove the lower bounds.\\
	
	\textbf{Case (i):} For the inequality \eqref{thefirstcasdaydic}, let $A = \|\pi_g\|_{\dhprn \to\dot{H}_d^q(\rn) }$. This means
	\[
	\left\|S_d(\pi_g(f))\right\|_{L^q(\rn)} \le A \|f\|_{H_d^p(\rn)},
	\]
	which is equivalent to
	\[
	\left\| \sum_{Q \in \mathcal{D}} \sum_{i \neq 0} |\langle f\rangle_{Q}|^2 \langle g, h^i_Q \rangle^2 \frac{\chi_Q}{|Q|} \right\|_{L^{\frac{q}{2}}(\rn)} \le A^2 \|f\|_{H_d^p(\rn)}^2.
	\]
	For each fixed $i \neq 0$, we have
	\[
	\left\| \sum_{Q \in \mathcal{D}} |\langle f\rangle_{Q}|^2 \langle g, h^i_Q \rangle^2 \frac{\chi_Q(x)}{|Q|} \right\|_{L^{\frac{q}{2}}(\rn)} \le A^2 \|f\|_{H_d^p(\rn)}^2,
	\]
	which matches the assumption \eqref{assumption} in Theorem \ref{mainthereom} with
	\[
	g_Q = \frac{\langle g, h^i_Q \rangle^2}{|Q|}, \quad s = 2, \quad \frac{1}{\frac{q}{2}} = \frac{1}{\frac{p}{2}} + \frac{1}{\frac{r}{2}}.
	\]
	Thus, by \eqref{conclusion ofthe maintheorem}, we get
	\[
	\left\| \sum_{Q \in \mathcal{D}} \frac{\langle g, h^i_Q \rangle^2 \chi_Q(x)}{|Q|} \right\|_{L^{\frac{r}{2}}(\rn)} \lesssim A^2.
	\]
	Summing over $i \neq 0$ and using the (quasi) triangle inequality, we obtain
	\[
	\left\|S_d(g)\right\|_{L^r(\rn)} = \|g\|_{\dot{H}_d^r(\rn)} \lesssim A,
	\]
	proving the claim for Case (i).\\
	
	\textbf{Case (ii):} For the inequality \eqref{thesecondcase dyadic}, let $A = \|\pi_g\|_{\dhprn \to \dot{H}_d^{p^*}(\rn)}$. We first consider the case where $g$ has a finite Haar expansion. For a dyadic cube $R$, choose $i \neq 0$ so that $f = |\hat{R}|^{\frac{1}{2}} h^i_{\hat{R}}$ is equal to 1 on $R$. Then
	\[
	\pi_g(f) = \sum_{Q \subseteq R} \sum_{j \neq 0} \langle g, h^j_Q \rangle h^j_Q + \sum_{Q \nsubseteq R} \sum_{j \neq 0} \langle f\rangle_{Q} \langle g, h^j_Q \rangle h^j_Q = g_1 + g_2,
	\]
	which implies
	\[
	\|g_1\|_{\dot{H}_d^{p^*}(\rn)} \le \|\pi_g(f)\|_{\dot{H}_d^{p^*}(\rn)} \le A \|f\|_{H_d^p(\rn)} = A |\hat{R}|^{\frac{1}{p}}\lesssim A |R|^{\frac{1}{p}}.
	\]
	Here, $g_1 = (g - \langle g \rangle_R) \chi_R$. Since $|g_1(x)| \le M_d(g_1)(x)$ for $x \in \rn$, by the maximal characterization of $\dot{H}_d^{p^*}(\rn)$, we have
	\[
	\|g_1\|_{L^{p^*}(\rn)} \le \|M_d(g_1)\|_{L^{p^*}(\rn)} \lesssim \|g_1\|_{\dot{H}^{p^*}(\rn)} \lesssim A |R|^{\frac{1}{p}},
	\]
	which is equivalent to
\begin{equation}\label{pstarocillation}
		\text{osc}_{p^*}(g, R) \lesssim A l(R)^{\alpha}.
\end{equation}
	Combining this with \eqref{johnnirenberg} yields $\dhlipalpharn{g} \lesssim A$, proving the claim for $g$ with a finite Haar expansion. To remove the restriction on $g$, let $N$ be a natural number and define
	\[
	g_N = \sum_{2^{-N} \le l(Q) \le 2^N} \sum_{j \neq 0} \langle g, h^j_Q \rangle h^j_Q.
	\]
	Note that $\|\pi_{g_N}\|_{\dhprn \to \dot{H}_d^{p^*}(\rn)} \le A$, which implies
	\[
	\dhlipalpharn{g_N} \lesssim A.
	\]
	For a fixed dyadic cube $R$, using \eqref{pstarocillation} and \eqref{johnnirenberg}, we get
	\[
	\|g'_N\|_{L^2(R, \frac{dx}{|R|})} \lesssim A l(R)^{\alpha},
	\]
	where
	 \[
	 g'_N = (g_N - \langle g_N \rangle_R) \chi_R = \sum_{Q \subseteq R} \sum_{j \neq 0} \langle g_N, h^j_Q \rangle h^j_Q.
	 \]
By the weak compactness of $L^2$, a subsequence of $g'_N$ converges to a function $g'_R$ with
	\[
	\|g'_R\|_{L^2(R, \frac{dx}{|R|})} \lesssim A l(R)^{\alpha}.
	\]
Since $\langle g'_N, h^i_Q \rangle$ converges to $\langle g'_R, h^i_Q \rangle$, we conclude that $g'_R$ coincides with $g$ on $R$. Hence, $g$ is a locally integrable function, and on each dyadic cube $R$, it satisfies $\text{osc}_2(g, R) \lesssim A l(R)^\alpha$, and thus,
	\[
	\dhlipalpharn{g} \lesssim A,
	\]
	which completes the proof.
\end{proof}

Now, we proceed with the proof of Theorem \ref{mainthereom}. To do this, we use the following lemma, which is similar to the construction in \cite{MR2157745} (Theorem 1.2.4). Here, it is more convenient to introduce a notation. For a family of nonnegative numbers $\left\{g_Q\right\}_{Q\in\mathcal{D}}$ and a dyadic cube $R$, we define
\[
(g|R)(x) := \sum_{Q\in\mathcal{D}(R)} g_Q \chi_Q(x).
\]

\begin{lemma}\label{sparslemma}
	Let $\left\{g_Q\right\}_{Q\in\mathcal{D}}$ be a family of nonnegative numbers, $Q_0$ a fixed dyadic cube, and $0 < \eta < 1$. Then, there exists an $\eta$-sparse family of dyadic cubes $\mathcal{C}$ in $Q_0$ with the property that:\\
	
	For each $Q \in \mathcal{C}$, there exists an integer $\lambda_Q$ such that
	\begin{align}
		&(g|Q_0) \lesssim \sum_{Q \in \mathcal{C}} 2^{\lambda_Q} \chi_Q \label{sparsedomination1} \\
		&|Q| \lesssim |\left\{ (g|Q) > 2^{\lambda_Q - 1} \right\}| \label{sparsedomination2}.
	\end{align}
\end{lemma}

\begin{proof}
Let 
\[
f_Q = (g|Q), \quad f_{P,Q} = f_Q - f_P = \sum_{P \subsetneq R \subseteq Q} g_R \chi_R, \quad P \in \mathcal{D}(Q),
\]
be as in Theorem \ref{operatorfree}. Then, we note that the condition 
\[
|f_{P,Q}| \le |f_Q| + |f_P|,
\]
holds, and therefore, an application of Theorem \ref{operatorfree} gives a family of cubes in $Q_0$, which is $\eta$-sparse and such that
\begin{equation}\label{1111}
	(g|Q_0) \lesssim \sum_{Q \in \mathcal{C}} \gamma_Q \chi_Q,
\end{equation}
where in the above
\begin{equation*}\label{dddddd}	
	\gamma_Q=(f_Q\chi_Q)^*(\frac{1-\eta}{2^{n+2}}|Q|)+(m_Q^{\#}f)^*(\frac{1-\eta}{2^{n+2}}|Q|).
\end{equation*}
	Here, the first thing to note is that $f_{P,Q}$ is constant on $P$ and thus $\text{osc}(f_{P,Q},P)=0$, which implies that $m_Q^{\#}f$ vanishes on $Q$. Now, for $Q\in\mathcal{C}$, if $(f_Q\chi_Q)^*(\frac{1-\eta}{2^{n+2}}|Q|)$ is zero it doesn't appear in \eqref{1111}, and we remove $Q$ from $\mathcal{C}$, and if not let $\lambda_Q$ be an integer such that
	\[
2^{\lambda_Q-1}<(f_Q\chi_Q)^*(\frac{1-\eta}{2^{n+2}}|Q|)\le 2^{\lambda_Q},
	\]
	then \eqref{sparsedomination1} holds, and from the definition of non-increasing rearrangement we must have
	\[
	\frac{1-\eta}{2^{n+2}}|Q|< |\{f_Q>2^{\lambda_Q-1}\}|=|\left\{(g|Q)>2^{\lambda_Q-1}\right\}|,
	\]
	which shows that \eqref{sparsedomination2} holds as well, and this completes the proof.
\end{proof}
We break the proof of Theorem \ref{mainthereom} into two parts, depending on whether \( 1<sp<\infty \), in which case our argument heavily relies on sparseness, or \( 0<sp \leq 1 \), where instead of sparseness, we take advantage of the sub-additivity of the \( L^{sp} \) norm.

\begin{proof}[Proof of Theorem \ref{mainthereom}]
Case (1). \(1 < sp < \infty\):\\

First, assume that there are only finitely many nonzero coefficients in \(\left\{g_R\right\}_{R \in \mathcal{D}}\), and let \(Q_0\) be a dyadic cube. Then, an application of Lemma \ref{sparslemma} with \(\eta = \frac{1}{2}\) gives us a \(\frac{1}{2}\)-sparse collection of dyadic cubes \(\mathcal{C}\), satisfying \eqref{sparsedomination1} and \eqref{sparsedomination2}. Now, observe that from Lemma \ref{lemmasparse} it follows that
\begin{equation}\label{estimateforg}
	\lVert (g|Q_0) \rVert_{L^r(\rn)}^r \lesssim \sum_{Q \in \mathcal{C}} 2^{r \lambda_Q} |Q|.
\end{equation}
Then, let
\[
T(f) = \sum_{R \in \mathcal{D}} |\avr{f}{R}|^s g_R \chi_R, \quad f = \sum_{Q \in \mathcal{C}} 2^{t \lambda_Q} \chi_Q, \quad t = \frac{r}{sp},
\]
and note that 
\[
\avr{f}{R} \ge 2^{t \lambda_Q}, \quad R \in \mathcal{D}(Q), \quad Q \in \mathcal{C},
\]
which implies that
\begin{equation}\label{replacinggbyT}
	\left\{ (g|Q) > 2^{\lambda_Q - 1} \right\} \subseteq \left\{ T(f) > 2^{(st + 1) \lambda_Q - 1} \right\} \cap Q.
\end{equation}
Also, another application of Lemma \ref{lemmasparse} gives us
\begin{equation}\label{estimateforf}
	\|f\|_{L^{sp}(\rn)} \simeq \left( \sum_{Q \in \mathcal{C}} 2^{r \lambda_Q} |Q| \right)^{\frac{1}{sp}}.
\end{equation}
Now, we proceed to estimate the right-hand side of \eqref{estimateforg}. To this aim, let us partition \(\mathcal{C}\) as
\[
\mathcal{C}'_k := \left\{ Q \in \mathcal{C} \mid \lambda_Q = k \right\}, \quad k \in \zed,
\]
which implies that
\begin{equation}\label{sumksumq}
	\sum_{Q \in \mathcal{C}} 2^{r \lambda_Q} |Q| = \sum_{k \in \zed} 2^{kr} \sum_{Q \in \mathcal{C}'_k} |Q|.
\end{equation}
Now, let \(\mathcal{C}''_k\) be the collection of maximal cubes in \(\mathcal{C}'_k\), and note that since this collection is \(\frac{1}{2}\)-sparse, we can estimate the last sum as
\begin{equation}\label{passigtomaximal}
	\sum_{Q \in \mathcal{C}'_k} |Q| \le 2 \sum_{Q \in \mathcal{C}'_k} |E_Q| \le 2 \left| \cup \left\{ Q \in \mathcal{C}'_k \right\} \right| = 2 \sum_{Q \in \mathcal{C}''_k} |Q|.
\end{equation}
Also, from the second property \eqref{sparsedomination2} of the cubes in \(\mathcal{C}\), together with \eqref{replacinggbyT}, we obtain
\[
\sum_{Q \in \mathcal{C}''_k} |Q| \lesssim \sum_{Q \in \mathcal{C}''_k} \left| \left\{ (g|Q) > 2^{k-1} \right\} \right| \le \sum_{Q \in \mathcal{C}''_k} \left| \left\{ T(f) > 2^{(st+1)k-1} \right\} \cap Q \right|, \quad k\in \zed,
\]
and then by noting that cubes in \(\mathcal{C}''_k\) are disjoint, we get the following estimate:
\[
\sum_{Q \in \mathcal{C}''_k} |Q| \lesssim \left| \left\{ T(f) > 2^{(st+1)k-1} \right\} \right|, \quad k \in \zed.
\]
Putting the above inequality together with \eqref{sumksumq} and \eqref{passigtomaximal}, we obtain
\[
	\sum_{Q \in \mathcal{C}} 2^{r \lambda_Q} |Q| \lesssim \sum_{k \in \zed} 2^{kr} \left| \left\{ T(f) > 2^{(st+1)k-1} \right\} \right|,
\]
which, after noting that \(st+1=\frac{r}{q}\), and using the layer cake formula implies that
\[
\sum_{Q \in \mathcal{C}} 2^{r \lambda_Q} |Q| \lesssim \int \sum_{2^{\frac{kr}{q}} < 2 T(f)(x)} 2^{kr} \, dx \lesssim \int_{\rn} T(f)^q.
\]
At the end, we use \eqref{assumption} and \eqref{estimateforf} to obtain
\[
\sum_{Q \in \mathcal{C}} 2^{r \lambda_Q} |Q| \lesssim \int_{\rn} T(f)^q \le A^q \|f\|_{L^{sp}}^{sq} \lesssim A^q \left( \sum_{Q \in \mathcal{C}} 2^{r \lambda_Q} |Q| \right)^{\frac{q}{p}},
\]
which, together with \eqref{estimateforg}, gives us
\[
\lrrnn{(g|Q_0)} \lesssim \left( \sum_{Q \in \mathcal{C}} 2^{r \lambda_Q} |Q| \right)^{\frac{1}{r}} \lesssim A.
\]
Now, since there are only finitely many nonzero terms in \(\left\{ g_R \right\}_{R \in \mathcal{D}}\), for \(2^n\) large dyadic cubes in each octant of \(\rn\), we have
\[
\sum_{R \in \mathcal{D}} g_R \chi_R = \sum_{i=1}^{2^n} (g|Q_i), \quad Q_i = [0, \pm 2^N]^n.
\] 
Then, applying the above inequality to each cube \(Q_i\), and using the (quasi) triangle inequality, we get
\[
\lrrnn{\sum_{R \in \mathcal{D}} g_R \chi_R} \lesssim A,
\]
and this proves the claim when there are only finitely many nonzero terms.
 To remove this restriction, let 
 \[
 (g_N)_R = \begin{cases*}
 	g_R & $2^{-N} \le l(R) \le 2^{N}$, \\
 	0 & \text{Otherwise},
 \end{cases*}
 \]
 Then from the fact that the assumption \eqref{assumption} still holds with \(A\), and the above inequality, we get 
 \[
 \lrrnn{\sum_{2^{-N} \le l(R) \le 2^{N}} g_R \chi_R} \lesssim A,
 \]
 which, after an application of Fatou's lemma, gives the desired conclusion and completes the proof of the first case.\\

Case(2). $0<sp\le1$. Here, we cannot use the function \(f\), constructed above, as a test function because functions in \(H_d^{sp}(\rn)\) must have lots of cancellations. However, in this case, the sparseness of the family is not necessary, and we may use sub-additivity which helps us to repeat a similar argument.\\

So, as in the previous case, suppose only finitely many terms in \(\{g_R\}_{R\in\mathcal{D}}\) are nonzero. This time, let 
\[
G := \sum_{R\in\mathcal{D}} g_R \chi_R,
\]
and for \(k \in \zed\), let \(\mathcal{C}_k\) be the collection of maximal dyadic cubes in \(\{G > 2^k\}\). From the layer cake formula we have
\begin{equation}\label{GGGG}
	\lrrnn{G}^r \simeq \sum_{k \in \zed} 2^{rk} |\{G > 2^k\}| = \sum_{k \in \zed} 2^{rk} \sum_{Q \in \mathcal{C}_k} |Q|.
\end{equation}
Now, just like in the previous case, we try to estimate the last sum. In order to do this, let \(\hat{\mathcal{C}}_k\) be the collection of maximal cubes in \(\{\hat{Q} \mid Q \in \mathcal{C}_k\}\), and for each \(Q' \in \hat{\mathcal{C}}_k\), let \(\tilde{\chi}_{Q'} = |Q'|^{\frac{1}{2}} h^i_{Q'}\) for some \(i \neq 0\). We note that this function is either \(+1\) or \(-1\) on the children of $Q'$, and belongs to \(H_d^{sp}(\rn)\) with \(\|\tilde{\chi}_{Q'}\|^{sp}_{H_d^{sp}(\rn)} = |Q'|\). Then, let
\[
T(f) = \sum_{R \in \mathcal{D}} |\avr{f}{R}|^s g_R \chi_R, \quad f = \sum_{k \in \zed} 2^{kt} \sum_{Q \in \hat{\mathcal{C}}_k} \tilde{\chi}_Q \quad t = \frac{r}{sp}.
\]
The above function, $f$, has the following two crucial properties:
\begin{align}
	&\|f\|^{sp}_{H_d^{sp}(\rn)} \lesssim \sum_{j \in \zed} 2^{rk} \sum_{Q \in \mathcal{C}_k} |Q|,\label{hspnormf}\\
	& |\avr{f}{R}| \gtrsim 2^{kt}, \quad R \in \mathcal{D}(Q), \quad Q \in \mathcal{C}_k, \quad k \in \zed.\label{theavrageoff}
\end{align}
To see the first one, note that from sub-additivity it follows
\[
\|f\|^{sp}_{H_d^{sp}(\rn)} \le \sum_{k \in \zed} 2^{ktsp} \sum_{Q' \in \hat{\mathcal{C}}_k} \|\tilde{\chi}_{Q'}\|^{sp}_{H_d^{sp}(\rn)} \le 2^n \sum_{k \in \zed} 2^{rk} \sum_{Q \in \mathcal{C}_k} |Q|.
\]
In order to see the second property, let \(Q \in \mathcal{C}_k\) and fix \(R \in \mathcal{D}(Q)\). Then, we decompose \(f\) as
\[
f(x) = \sum_{j \in \zed} 2^{jt} \sum_{\substack{Q' \in \hat{\mathcal{C}}_j \\ Q' \subseteq R}} \tilde{\chi}_{Q'}(x) + \sum_{j \in \zed} 2^{jt} \sum_{\substack{Q' \in \hat{\mathcal{C}}_j \\ R \subsetneq Q'}} \tilde{\chi}_{Q'}(x) = f_1(x) + f_2(x), \quad x \in R.
\]
Now, because of the cancellation of the functions \(\tilde{\chi}_{Q'}\), we have \(\avr{f_1}{R} = 0\). Furthermore, since $\tilde{\chi}_{Q'}$, is either \(+1\) or \(-1\) on the children of $Q'$, \(f_2\) is constant on \(R\). Next, we note that for each \(j \in \zed\), \(R\) is contained in at most one cube in \(\hat{\mathcal{C}}_j\), and when this inclusion is strict, the contribution of each term in the right-hand sum on \(R\) is either \(+2^{jt}\) or \(-2^{jt}\). Therefore, 
\[
f_2 = \sum_{\substack{Q' \in \hat{\mathcal{C}}_j \\ R \subsetneq Q'}} \pm 2^{jt} \simeq \pm2^{jl},
\]
where in the above \(l\) is the largest \(j \in \zed\) such that \(R\) is strictly contained in a cube \(Q' \in \hat{\mathcal{C}}_j\). Then, since \(R \subseteq Q \subsetneq \hat{Q}\), and \(\hat{Q}\) is contained in a cube in \(\hat{\mathcal{C}}_k\), we conclude that \(l \ge k\), and this shows that
\[
|\avr{f}{R}| = |\avr{f_2}{R}|\simeq 2^{lt}\ge 2^{kt},
\]
which proves the second property of \(f\). Next, we proceed to estimate the measure of the level sets of \(G\) in terms of \(T(f)\). So, let \(Q \in \mathcal{C}_k\). We claim that
\begin{equation}\label{Tclaim}
	Q \subseteq \{T(f) \gtrsim 2^{(st+1)k}\},
\end{equation}
and in order to see this, we consider two cases: either \(g_Q > 2^{k-1}\) or \(g_Q \le 2^{k-1}\). In the first case, the claim follows from \eqref{theavrageoff} as we have
\begin{equation*}
	T(f) \ge |\avr{f}{Q}|^s g_Q \chi_Q \gtrsim 2^{(st+1)k} \chi_Q.
\end{equation*}
For the second case, we note that by the maximality of \(Q\) in \(\{G > 2^k\}\), we have
\begin{equation}\label{2kinequality}
	\sum_{Q \subseteq R} g_R \chi_R > 2^k \chi_Q,
\end{equation}
and since we assume \(g_Q \le 2^{k-1}\), we must have
\begin{equation*}
	\sum_{\hat{Q} \subseteq R} g_R \chi_R > 2^{k-1} \chi_{\hat{Q}},
\end{equation*}
So for \(Q'\) a maximal cube in \(\mathcal{C}_{k-1}\), we have \(\hat{Q} \subseteq Q'\), and then the maximality of \(Q'\) implies that
\begin{equation}\label{2k-1ineq}
	\sum_{\hat{Q'} \subseteq R} g_R \chi_R \le 2^{k-1} \chi_{\hat{Q'}}, \quad Q \subsetneq \hat{Q} \subseteq Q' \subsetneq \hat{Q'}.
\end{equation}
Subtracting \eqref{2k-1ineq} from \eqref{2kinequality}, we obtain
\begin{equation*}
	\sum_{Q \subseteq R \subseteq Q'} g_R \chi_R > 2^{k-1} \chi_Q.
\end{equation*}
Now, since \(Q' \in \mathcal{C}_{k-1}\), it follows from \eqref{theavrageoff} that we have
\[
|\avr{f}{R}| \gtrsim 2^{t(k-1)}, \quad Q \subseteq R \subseteq Q',
\]
which together with the above inequality implies that 
\[
T(f) \ge \sum_{Q \subseteq R \subseteq Q'} |\avr{f}{R}|^s g_R \chi_R \gtrsim 2^{ts(k-1)} \sum_{Q \subseteq R \subseteq Q'} g_R \chi_R \gtrsim 2^{k(st+1)} \chi_Q,
\]
which proves our claim in \eqref{Tclaim}. Now, it follows from \eqref{Tclaim} and disjointness of cubes in \(\mathcal{C}_k\) that
\[
\sum_{Q \in \mathcal{C}_k} |Q| \le |\{T(f) \gtrsim 2^{k(st+1)}\}|, \quad k \in \zed,
\]
and the rest of the proof follows the same line as in the previous case. Inserting the above estimate in the right-hand side of \eqref{GGGG} gives us
\[
\sum_{k \in \zed} 2^{rk} \sum_{Q \in \mathcal{C}_k} |Q| \le \sum_{k \in \zed} 2^{rk} |\{T(f) \gtrsim 2^{k(st+1)}\}| \lesssim \int_{\rn} T(f)^q \lesssim A^q \|f\|^{sq}_{H_d^{sp}(\rn)},
\]
and then using \eqref{hspnormf} and \eqref{GGGG} we obtain
\[
\lrrnn{G} \lesssim A,
\]
which is the desired result when there are only finitely many nonzero terms in \(\{g_R\}_{R \in \mathcal{D}}\), and then the limiting argument presented at the end of the previous case extends the result to the general case, and this completes the proof of this case and Theorem \ref{mainthereom}.
\end{proof}

We conclude this section by giving an example which shows that for \(0 < q < 1\), the operator norm of the formal adjoint of a dyadic paraproduct \(\pi_g\) can be much smaller than the norm of the operator itself. This happens because of the fact that the Hahn-Banach theorem fails for \(\dhqrn\) \cite{MR0259579,MR2163588}.\\

\textbf{Example}. 
Let \(g\) be a dyadic distribution on \(\mathbb{R}\). Then the formal adjoint of \(\pi_g\) is given by
\[
\pi_g^t(f) := \sum_{I \in \mathcal{D}(\mathbb{R})} \avr{f, h_I}{} \avr{g, h_I}{} \frac{\chi_I}{|I|}.
\]
Now let 
\[
g = \sum_{\substack{I \in \mathcal{D}[0,1] \\ |I| = 2^{-l}}} |I|^{\frac{1}{2}} h_I, \quad \frac{1}{q} = \frac{1}{p} + \frac{1}{r}, \quad 0 < q < 1 < p < \infty,
\]
and note that \(S(g) = \chi_{[0,1]}\), so we have \(\|g\|_{\dot{H}_d^r(\mathbb{R})} = 1\), and Theorem \ref{dyadictheorem} implies that
\[
\|\pi_g\|_{H_d^p(\mathbb{R}) \to \dot{H}_d^q(\mathbb{R})} \simeq 1.
\]
In fact, here we do not need to use Theorem \ref{dyadictheorem} to conclude this. This can be simply proven by testing the operator \(\pi_g\) on \(g\) itself. However, the norm of the adjoint operator \(\pi_g^t: \dot{\Lambda}^{\frac{1}{q}-1}(\mathbb{R}) \to (H_d^p(\mathbb{R}))'\) can be estimated as follows: Let \(\|f\|_{\dot{\Lambda}^{\frac{1}{q}-1}} = 1\), and note that
\[
|\avr{f, h_I}{}| \leq \text{osc}_1(f, I) |I|^{\frac{1}{2}} \leq |I|^{\frac{1}{q} - \frac{1}{2}},
\] 
which implies that
\[
|\pi_g^t(f)| \leq \sum_{\substack{I \in \mathcal{D}[0,1] \\ |I| = 2^{-l}}} |I|^{\frac{1}{q} - 1} \chi_I = 2^{l\left(1 - \frac{1}{q}\right)} \chi_{[0,1]},
\]
and thus
\[
\|\pi_g^t(f)\|_{(H_d^p(\mathbb{R}))'} \simeq \|\pi_g^t(f)\|_{L^{p'}(\mathbb{R})} \leq 2^{l\left(1 - \frac{1}{q}\right)},
\]
which shows that for large \(l\), the operator norm of \(\pi_g^t\) is much smaller than that of \(\pi_g\).

\section{Fourier paraproducts}
In this section, we show that similar results hold for the operator $\Pi_{g,\varphi}$. Here, we are faced with many error terms and are forced to assume more than merely the boundedness of $\Pi_{g,\varphi}$. To resolve these difficulties, instead of merely assuming that $\Pi_{g,\varphi}$ is bounded from $\hprn$ to $\dot{H}^q(\rn)$, we assume that the sublinear operator 
\begin{equation}\label{sgdefinition}
	\mathcal{S}_{g,\varphi}(f):=\left(\sum_{j\in\zed}|\varphi_{2^{-j}}*f\Delta_j(g)|^2\right)^{\frac{1}{2}},
\end{equation} 
is bounded from $\hprn$ to $\lqrn$. When $\hat{\varphi}$ has compact support and satisfies the restriction mentioned in Theorem (B), this extra assumption is equivalent to assuming that the operators $\Pi_{i,g,\varphi}$, mentioned in \eqref{pigdefinition}, are all bounded. The reason for this is that
\[
\mathcal{S}_{g,\varphi}(f)(x)\simeq \sum_{0\le i<m} S_\theta(\Pi_{i,g,\varphi}f)(x), \quad x\in\rn,
\]
and thus
\[
\|\mathcal{S}_{g,\varphi}\|_{\hprn\to\lqrn}\simeq \sum_{0\le i<m} \|\Pi_{i,g,\varphi}\|_{\hprn\to\dot{H}^q(\rn)}.
\]
Now, we state the main result of this section.

\begin{theorem}\label{continuousparp}
	Let $\psi$ be a Schwartz function as in \eqref{psiconditins} and \eqref{partionofunity}, and let $\varphi$ be a Schwartz function whose Fourier transform is supported in a ball with radius $\textbf{a}'<\textbf{a}$, where $\textbf{a}$ is as in \eqref{psiconditins}, and equal to $1$ in a smaller neighborhood of the origin. Also, let $g$ be a tempered distribution on $\rn$, and let the sublinear operator $\mathcal{S}_{g,\varphi}$ be as in \eqref{sgdefinition}. Then we have
	\begin{align}
		&\tag{I}\|\mathcal{S}_{g,\varphi}\|_{\hprn\to L^{p^*}(\rn)}\simeq \|g\|_{\ddot{\Lambda}^\alpha(\rn)}, \quad \frac{1}{p^*}=\frac{1}{p}-\frac{\alpha}{n},\quad 0<\alpha p<n, \quad 0<p<\infty\label{hpstarhpclaim},\\
		&\tag{II}\|\mathcal{S}_{g,\varphi}\|_{\hprn\to\lprn}\simeq \|g\|_{\dot{BMO}(\rn)}, \quad 0<p<\infty\label{hphpclaim},\\
		&\tag{III}\|\mathcal{S}_{g,\varphi}\|_{\hprn\to\lqrn}\simeq \|g\|_{\dot{H}^r(\rn)}, \quad \frac{1}{q}=\frac{1}{p}+\frac{1}{r}, \quad 0<q,p,r<\infty\label{hphqclaim}.
	\end{align}
\end{theorem}

In the above, the upper bounds for the operator norm of $\mathcal{S}_{g,\varphi}$ follow directly from Theorem (B) and the above discussion, and it remains only to prove the lower bounds. As we mentioned before, here we are dealing with some error terms that, at the end of the proofs, have to be absorbed into the left-hand side. To this aim, we bring the following simple lemma.
\begin{lemma}\label{convolutionelemma}
	Let $\phi$ be a Schwartz function, $g \in \bmorn$, and $Q$ a cube. Then we have
	\[
	|\phi_t*(g - \avr{g}{Q})|(x) \lesssim (1 + |\log t l(Q)^{-1}|) \bmornn{g}, \quad x \in Q, \quad t > l(Q).
	\]
\end{lemma}

\begin{proof}
	Since $\phi$ is a Schwartz function, we may assume
	\[
	|\phi(y)| \le \frac{C(\phi)}{(1 + |y|)^{n + \delta}}, \quad y \in \rn, \quad \delta > 0.
	\]
	Now, we have
	\begin{align*}
		&|\phi_t*(g - \avr{g}{Q})|(x)\lesssim \int_{\rn} \frac{t^{-n}}{(1 + |\frac{y - x}{t}|)^{n + \delta}} |g(y) - \avr{g}{Q}| dy\\
		&\lesssim \int_{Q} t^{-n} |g(y) - \avr{g}{Q}|dy
		+ \sum_{k \ge 0} \int_{2^{k+1}Q \backslash 2^kQ} \frac{t^{-n}}{(1 + t^{-1} l(Q) 2^k)^{n + \delta}} |g(y) - \avr{g}{Q}|dy,
	\end{align*}
	where in the above we used the fact that for $x \in Q$ and $y \in 2^{k+1}Q\backslash 2^kQ$ we have 
	\[
	|x - y| \simeq 2^k l(Q).
	\]
	Next, we note that
	\[
	\int_{2^{k+1}Q} |g(y) - \avr{g}{Q}| \lesssim \bmornn{g} (k + 1) |2^{k+1}Q|, \quad k \ge 0,
	\]
	which, after plugging in the above, gives us
	\begin{equation}\label{hhhhhh}
		|\phi_t*(g - \avr{g}{Q})|(x) \lesssim \bmornn{g} (t^{-1} l(Q))^n \left(1 + \sum_{k \ge 0} \frac{2^{kn} k}{(1 + t^{-1} l(Q) 2^k)^{n + \delta}}\right).
	\end{equation}
	Then, we estimate the last sum as
	\[
	\sum_{k \ge 0} \frac{2^{kn} k}{(1 + t^{-1} l(Q) 2^k)^{n + \delta}} \le \sum_{t^{-1} l(Q) 2^k \le 1} 2^{kn} k + \sum_{t^{-1} l(Q) 2^k > 1} \frac{2^{-k \delta} k}{(t^{-1} l(Q))^{n + \delta}}.
	\]
	Noting that because of the geometric factor, each sum in the above is dominated by its largest term, we obtain
	\[
	\sum_{k \ge 0} \frac{2^{kn} k}{(1 + t^{-1} l(Q) 2^k)^{n + \delta}} \lesssim 1 + (t l(Q)^{-1})^n |\log t l(Q)^{-1}|,
	\]
	which, together with \eqref{hhhhhh} and our assumption that $t > l(Q)$, implies
	\[
	|\phi_t*(g - \avr{g}{Q})|(x) \lesssim \bmornn{g} (1 + |\log t l(Q)^{-1}|),
	\]
	and this completes the proof.
\end{proof}

We break the proof of Theorem \ref{continuousparp} into two parts, and first we prove \eqref{hpstarhpclaim} and \eqref{hphpclaim}.

\begin{proof}[Proofs of \eqref{hpstarhpclaim} and \eqref{hphpclaim}]
	Let $A = \|\mathcal{S}_{g,\varphi}\|_{\hprn \to L^{p^*}(\rn)}$, and let $c$ be such that
	\[
	\hat{\varphi}(\xi) = 1, \quad |\xi| \le c.
	\]
	We take a Schwartz function $f$ with
	\[
	\text{supp}(\hat{f}) \subseteq B_c \setminus B_{\frac{c}{2}}, \quad f(0) = \int \hat{f} = 1,
	\]
	and for $j \in \zed$ and $x_0 \in \rn$, consider the function $f_{j,x_0} = \tau^{x_0} \delta^{2^{-j}} f$, whose Fourier transform is supported in $B_{c2^j} \setminus B_{c2^{j-1}}$ and $f_{j,x_0}(x_0) = 1$. Now, since $\hat{\varphi}$ is equal to $1$ on $B_c$, we have $\varphi_{2^{-j}} * f_{j,x_0} = f_{j,x_0}$. Therefore, from the boundedness of $\mathcal{S}_{g,\varphi}$, it follows that
	\[
	\|\varphi_{2^{-j}} * f_{j,x_0} \Delta_j g\|_{L^{p^*}(\rn)} \le A \hprnn{f_{j,x_0}}.
	\]
Then, since $\hprnn{f_{j,x_0}} \simeq 2^{-j \frac{n}{p}}$, and $\varphi_{2^{-j}} * f_{j,x_0} = f_{j,x_0}$, we must have
	\[
	\|f_{j,x_0} \Delta_j g\|_{L^{p^*}(\rn)} \lesssim A 2^{-j \frac{n}{p}}.
	\]
	At this point, we note that the Fourier transform of the function $f_{j,x_0} \Delta_j g$ is supported in a ball of radius $\simeq 2^j$. Therefore, using the Plancherel-Polya-Nikolskij inequality \eqref{polyainequalitylplq}, we obtain
	\[
	\linftyrnn{f_{j,x_0} \Delta_j g} \lesssim 2^{j \frac{n}{p^*}} \|f_{j,x_0} \Delta_j g\|_{L^{p^*}(\rn)} \lesssim A 2^{j n (\frac{1}{p^*} - \frac{1}{p})} = A 2^{-j \alpha},
	\]
	and since $f_{j,x_0}(x_0) = 1$, we must have
	\[
	|\Delta_j g(x_0)| \lesssim A 2^{-j \alpha}, \quad x_0 \in \rn, \quad j \in \zed,
	\]
	which shows that
	\[
	\|g\|_{\ddot{\Lambda}^\alpha(\rn)} \simeq \sup_{j \in \zed} 2^{j \alpha} \linftyrnn{\Delta_j g} \lesssim A.
	\]
	This proves \eqref{hpstarhpclaim}. \\
	
	Now we proceed to the proof of \eqref{hphpclaim}. To this end, let $A = \|\mathcal{S}_{g,\varphi}\|_{\hprn \to \lprn}$, and choose a large integer $m$ such that the Fourier transforms of any two consecutive terms in
	\[
	\sum_{j \in m \zed + i} \varphi_{2^{-j}} *f \Delta_j(g), \quad 0 \le i < m,
	\]
	are at sufficiently large distances from each other. Also, for a fixed natural number $N$, let
	\begin{equation}\label{giNPiN}
		g_{i,N} = \sum_{\substack{j \in m \zed + i \\ |j| \le N}} \Delta_j g, \quad P_{i,N}(f) = \sum_{\substack{j \in m \zed + i \\ |j| \le N}} \varphi_{2^{-j}} * f \Delta_j g, \quad 0 \le i < m.
	\end{equation}
	From now on, we fix $i$, and for simplicity of notation, set $g' = g_{i,N}$ and $P = P_{i,N}$. We note that for a suitable choice of $\theta$, we have
	\[
	\Delta_l^\theta(g') = \delta_{j,l} \Delta_j g, \quad j \in m \zed + i, \quad l \in \zed,
	\]
	and thus we may replace $\Delta_j^\psi g$ with $\Delta_j^\theta(g')$. Therefore, we have
	\begin{equation}\label{g'Pf}
		g' = \sum_{j \in \zed} \Delta_j^\theta(g'), \quad P(f) = \sum_{j \in \zed} \varphi_{2^{-j}} * f \Delta_j^\theta(g'),
	\end{equation}
	and since there are only finitely many nonzero terms in the above expression, we have
	\[
	\|P\|_{\hprn \to \hprn} \simeq \|P\|_{\hprn \to \dot{H}^p(\rn)} \lesssim A.
	\]
	We now turn to estimating the mean oscillation of $g'$ over a cube. Before doing so, we must ensure that $g'$ belongs to $\bmorn$. To see this, we note that the previous argument for \eqref{hpstarhpclaim} still holds with $\alpha = 0$. Thus, $\linftyrnn{\Delta_j g} \lesssim A$, and since $g'$ is a finite sum of such terms, it is bounded as well and hence belongs to $\bmorn$.\\
	
So,	let $Q$ be a cube with $2^{-k} \le l(Q) < 2^{-k+1}$, and $x_0$ be its center. Also, let $k_0$ be a large number to be determined later, and take the function $h = f_{k - k_0, x_0}$, where $f$ is as in the previous case. Note that $\hat{h}$ is supported on $B_{c2^{k - k_0}} \setminus B_{c2^{k - k_0 - 1}}$. Then, since the Fourier transform of $\varphi_{2^{-j}}$ is equal to $1$ on $B_{c2^j}$ and vanishes outside of $B_{\textbf{\textit{a}}' 2^j}$, we conclude that for a sufficiently large choice of $m$ depending only on $c$ and $\textbf{\textit{a}}'$, we have
	\[
	\varphi_{2^{-j}} * h = h, \quad j \ge k - k_0, \quad \varphi_{2^{-j}} * h = 0, \quad j < k - k_0, \quad j \in m \zed + i.
	\]
	This observation implies that
	\[
	P(h) = h \sum_{k - k_0 \le j} \Delta_j^\theta(g') = h (g' - \sum_{j \le k - k_0 - 1} \Delta_j^\theta(g')) = h (g' - S_{k - k_0 - 1}^\theta(g')).
	\]
	Now, let $\Theta$ be the kernel function of the above partial sum operator. Then, since $\int \Theta = 1$, we may write
	\[
	g' - S_{k - k_0 - 1}^\theta(g') = g' - \avr{g'}{Q} - \Theta_{2^{-k + k_0 + 1}} * (g' - \avr{g'}{Q}),
	\]
	and thus we obtain
	\[
	P(h) = h \left(g' - \avr{g'}{Q} - \Theta_{2^{-k + k_0 + 1}} * (g' - \avr{g'}{Q}) \right).
	\]
	Also, since $h(x_0) = 1$, we can decompose the above sum as
	\[
	P(h) = g' - \avr{g'}{Q} - \Theta_{2^{-k + k_0 + 1}} * (g' - \avr{g'}{Q}) + (h - h(x_0)) \left(g' - \avr{g'}{Q} - \Theta_{2^{-k + k_0 + 1}} * (g' - \avr{g'}{Q}) \right),
	\]
	which is equivalent to
	\begin{align*}
		&g' - \avr{g'}{Q}= P(h) + \Theta_{2^{-k + k_0 + 1}} * (g' - \avr{g'}{Q}) - (h - h(x_0)) (g' - \avr{g'}{Q}) \\
		&+ (h - h(x_0)) \Theta_{2^{-k + k_0 + 1}} * (g' - \avr{g'}{Q}) = P(h) + E_1 - E_2 + E_3.
	\end{align*}
	This implies that
	\begin{equation}\label{mean oscilationg}
		\text{osc}_p(g', Q) \lesssim \text{osc}_p(P(h), Q) + \sum_{i=1}^{3} \text{osc}_p(E_i, Q),
	\end{equation}
	where
	\[
	E_1 = \Theta_{2^{-k + k_0 + 1}} * (g' - \avr{g'}{Q}), \quad E_2 = (h - h(x_0)) (g' - \avr{g'}{Q}), \quad E_3 = (h - h(x_0)) \Theta_{2^{-k + k_0 + 1}} * (g' - \avr{g'}{Q}).
	\]
	Now we estimate the $p$-mean oscillation of each term. For the first term, we have
	\[
	\text{osc}_p(P(h), Q) \lesssim |Q|^{-\frac{1}{p}} \lprnn{P(h)} \lesssim 2^{k \frac{n}{p}} \hprnn{P(h)} \lesssim A 2^{k \frac{n}{p}} \hprnn{h}.
	\]
	Then, since $\hprnn{h} \simeq 2^{\frac{n(k_0 - k)}{p}}$, which follows from the fact that $h = \tau^{x_0} \delta^{2^{k_0 - k}}(f)$, the above inequality implies
	\begin{equation}\label{phestiamte}
		\text{osc}_p(P(h), Q) \lesssim 2^{\frac{n k_0}{p}} A.
	\end{equation}
	For the second term, we have
	\[
	\text{osc}_p(E_1, Q) \lesssim l(Q) \sup_{x \in Q} \left| \nabla \Theta_{2^{-k + k_0 + 1}} * (g' - \avr{g'}{Q})(x) \right|,
	\]
	where we used the mean value theorem. Since for any $t > 0$, $\nabla \Theta_t = t^{-1} (\nabla \Theta)_t$, we may write
	\begin{equation}\label{previousinequality}
		\text{osc}_p(E_1, Q) \lesssim 2^{k - k_0 - 1} l(Q) \sup_{x \in Q} \left| (\nabla \Theta)_{2^{-k + k_0 + 1}} * (g' - \avr{g'}{Q})(x) \right|,
	\end{equation}
	and by applying Lemma \ref{convolutionelemma} to $\nabla \Theta$, we get
	\[
	\sup_{x \in Q} \left| (\nabla \Theta)_{2^{-k + k_0 + 1}} * (g' - \avr{g'}{Q})(x) \right| \lesssim \left(1 + \log 2^{-k + k_0 + 1} l(Q)^{-1} \right) \bmornn{g'},
	\]
	which, noting that $l(Q) \simeq 2^{-k}$, together with \eqref{previousinequality}, gives
	\begin{equation}\label{E1estimate}
		\text{osc}_p(E_1, Q) \lesssim 2^{-k_0} k_0 \bmornn{g'}.
	\end{equation}
	To estimate the third term, we note that
	\[
	\text{osc}_p(E_2, Q) \lesssim \avr{|E_2|^p}{Q}^{\frac{1}{p}} \le \sup_{x \in Q} |h(x) - h(x_0)| \text{osc}_p(g', Q) \lesssim l(Q) \sup_{x \in Q} |\nabla h(x)| \text{osc}_p(g', Q),
	\]
	and since $\linftyn{\nabla h} \simeq 2^{k - k_0}$, we obtain
	\begin{equation}\label{E2estimate}
		\text{osc}_p(E_2, Q) \lesssim 2^{-k_0} \text{osc}_p(g', Q) \lesssim 2^{-k_0} \bmornn{g'}.
	\end{equation}
	Finally, for the last term, we have
	\[
	\text{osc}_p(E_3, Q) \le2 \sup_{x \in Q} |h(x) - h(x_0)| \sup_{y \in Q} \left| \Theta_{2^{-k + k_0 + 1}} * (g' - \avr{g'}{Q})(y) \right|,
	\]
	which, using the mean value theorem for $h$ and applying Lemma \ref{convolutionelemma}, implies
	\begin{equation}\label{E3estimate}
		\text{osc}_p(E_3, Q) \lesssim l(Q) 2^{k - k_0} \left(1 + \log 2^{-k + k_0 + 1} l(Q)^{-1} \right) \bmornn{g'} \lesssim 2^{-k_0} k_0 \bmornn{g'}.
	\end{equation}
	Putting \eqref{mean oscilationg}, \eqref{phestiamte}, \eqref{E1estimate}, \eqref{E2estimate}, and \eqref{E3estimate} together, we obtain
	\[
	\text{osc}_p(g', Q) \lesssim 2^{\frac{n k_0}{p}} A + k_0 2^{-k_0} \bmornn{g'},
	\]
	and taking the supremum over all cubes gives
	\[
	\bmornn{g'} \lesssim 2^{\frac{n k_0}{p}} A + k_0 2^{-k_0} \bmornn{g'}.
	\]
Now, by choosing $k_0$ large enough and noting that we already know $g'$ belongs to $\bmorn$, we conclude that $\bmornn{g'} \lesssim A$. Finally, recall that
	\[
	g' = g_{i,N} = \sum_{\substack{j \in m \zed + i \\ |j| \le N}} \Delta_j g,
	\]
	and we have
	\[
	g_N = \sum_{|j| \le N} \Delta_j g = \sum_{0 \le i < m} g_{i,N}.
	\]
	Thus, this sequence must be bounded in $\bmorn$, and the Banach-Alaoglu theorem implies that there exists a subsequence converging in the weak* topology of $\bmorn$ to a function $G$ with $\bmoronen{G} \lesssim A$. Since the sequence $g_N$ converges in the space of distributions modulo polynomials to $g$, we conclude that $g$ is equal to $G$ modulo a polynomial, which means that there exists a polynomial $U$ such that
	\[
	\bmornn{g - U} \lesssim A,
	\]
	and this proves \eqref{hphpclaim}.
\end{proof}

\begin{remark}
	We note that the boundedness of the operator $\mathcal{S}_{g,\varphi}$ on $L^2(\rn)$ is equivalent to the statement that the measure
	\[
	d\mu(x,t) = \sum_{j \in \zed} |\Delta_j g|^2 \, dx \, d\delta_{2^{-j}}(t),
	\]
	is a Carleson measure. For $p = 2$, our result can be rephrased as: if the above measure is Carleson, then $g$, belongs to $\dot{BMO}(\rn)$. As far as we know, the previous proof of this fact uses the assumption on $g$, and directly shows that modulo a polynomial $g$ lies in the dual of $\honern$. Fefferman's duality theorem then implies that $g \in \dot{BMO}(\rn)$\cite{MR1232192} (p. 161). However, here we only used the fact that, to estimate the $\bmorn$ norm, we may use any $p$-mean oscillation, which follows from the John-Nirenberg lemma, and the fact that every bounded sequence in $\bmorn$ has a subsequence converging in the topology of distributions modulo polynomials.
\end{remark}

We now proceed to the proof of \eqref{hphqclaim}, and in order to do so, we need a series of lemmas.
\begin{lemma}\label{testfunctionlemma}
Let $\varphi$ be a Schwartz function with $\int \varphi = 1$, and let $B_1$ be the unit ball in $\rn$. Then, for $\alpha \ge 2$ and $0 < p <\infty$, there exists a function $\tilde{\chi}$ supported in a ball with radius $c(\alpha, p, \varphi)$ such that
\begin{align}
	&|\varphi_t * \tilde{\chi}| > \frac{1}{3} \chi_{B_1}, \quad t \le \alpha \label{chitildisbig},\\
	&\hprnn{\tilde{\chi}} \le c'(\varphi, \alpha, p). \label{atombound}
\end{align}
\end{lemma}

\begin{proof}
First, note that if $\varphi$ is complex valued then the integral of its real part is $1$, and therefore without loss of generality we may assume that $\varphi$ is real valued. Then, we construct an atom with a large amount of cancellation that satisfies the required conditions. To this aim, pick a large number $M$ such that
	\[
	\int_{|x|\ge M}|\varphi|\le \frac{1}{3}, \quad M>\frac{\alpha}{2},
	\]
	which implies that for $t \le \alpha$, we have
	\[
	\int_{|x|\ge \alpha M}|\varphi_t|(x) \, dx = \int_{|x|\ge \frac{\alpha}{t}M}|\varphi|(x) \, dx \le \frac{1}{3}.
	\]
Also, for $x \in B_1$, we have
\[
\varphi_t * \chi_{B_{\alpha M}}(x) = \int_{\rn} \varphi_t(x - y) \chi_{B_{\alpha M}}(y) \, dy = \int_{B_{\alpha M}(x)} \varphi_t(z) \, dz = 1 - \int_{B_{\alpha M}^c(x)} \varphi_t(z) \, dz,
\]
and since $B_{\alpha M}^c(x) \subseteq B_{M}^c(0)$ (because $\alpha \ge 2$), we have
\begin{equation}\label{greaterthantwothird}
	\varphi_t * \chi_{B_{\alpha M}}(x) \ge \frac{2}{3}, \quad x \in B_1.
\end{equation}
Now, take a natural number $N > n(\frac{1}{p} - 1)$, and let $B'$ be a ball of radius $1$ at distance $D$ from the origin, say $B' = B_1(2De)$ for some unit vector $e$. Then, we choose $P$, a polynomial of degree at most $N$, and set
\[
\tilde{\chi}(x) = \chi_{B_{\alpha M}}(x) + P(x) \chi_{B'}(x).
\]
To make $\tilde{\chi}$ into an atom, we need to find $P$ such that
\begin{equation}\label{Cancellation}
	\int \tilde{\chi}(x) Q(x) \, dx = 0,\quad Q\in \mathbb{P}_N,
\end{equation}
where $\mathbb{P}_N$ is the space of real valued polynomials in $n$ variables with degree no more than $N$. Additionally, $P$ has to be chosen such that $P\chi_{B'}$ has a small contribution to $\varphi_t * \tilde{\chi}$ on the unit ball, meaning that
\begin{equation}\label{lessthanthird}
	|\varphi_t * (P \chi_{B'})(x)| \le \frac{1}{3}, \quad  x\in B_1, \quad t\le \alpha.
\end{equation}
To achieve the first task, consider the inner product on $\mathbb{P}_N$, defined as 
\[
\left\langle Q_1, Q_2\right\rangle:= \int_{B_1}Q_1(x)Q_2(x)\, dx,
\]
and pick an orthonormal basis $\{Q_\beta\}_{|\beta|\le N}$ with respect to this inner product. Then, translate these polynomials to the center of $B'$, and set
\begin{equation}\label{coefqprime}
	Q_\beta'=\tau^{2De}Q_\beta, \quad |\beta|\le N,
\end{equation}
which gives an orthonormal basis for $\mathbb{P}_N$, equipped with the new inner product
\[
\left\langle Q_1, Q_2\right\rangle':= \int_{B'}Q_1(x)Q_2(x)\, dx.
\]
Now, the cancellation condition \eqref{Cancellation} becomes
\[
\left\langle P, Q_\beta'\right\rangle'= -\int_{B_{\alpha M}}Q_\beta', \quad |\beta|\le N,
\]
which has a unique solution $P$ defined as
\[
P(x)= \sum_{|\beta|\le N} \left\langle P, Q_\beta'\right\rangle' Q_\beta'(x)=\sum_{|\beta| \le N} c_\beta x^{\beta}.
\]
Then, to see why \eqref{lessthanthird} holds note that \eqref{coefqprime} implies that for each $\beta$, the coefficients of $Q_\beta'$, grow no more than a constant (which depends only on $n$, and $N$) times $D^N$, and therefore
\[
|\left\langle P, Q_\beta'\right\rangle'|\le \int_{B_{\alpha M}}|Q_\beta'|\le C(n, N, \alpha, M)D^N,
\]
which implies that
\[
|c_\beta| \le C'(n,N,,\alpha,M) D^{2N}, \quad D > 1, \quad |\beta|\le N.
\]
On the other hand, the Schwartz function $\varphi$ decays faster than any polynomial, so for a constant $C(\varphi, N)$, we have
\[
|\varphi(x)| \le C(\varphi, N) \frac{1}{(1 + |x|)^{3N + n}}, \quad x\in\rn.
\]
Now, for $x \in B_1$ we write
\[
|\varphi_t * (P \chi_{B'})(x)| \le \int_{B'} |\varphi_t(x - y)| |P(y)| \, dy,
\]
and we note that since $x \in B_1$ and $y \in B'$, we have $|x - y| \simeq D$. Also, since $P$ is a polynomial of degree at most $N$, we have
\[
|P(y)| \le C(n,N, \alpha, M) D^{3N}, \quad y \in B',
\]
which implies that
\[
|\varphi_t * (P \chi_{B'})(x)| \le \int_{B'} \varphi_t(x - y) |P(y)| \, dy \le C(n, \varphi, N, \alpha, M) D^{3N} t^{-n} \int_{B'} \frac{1}{|t^{-1} D|^{3N + n}}.
\]
Then, it is enough to recall that $|B'|=1$, and $t \le \alpha$ to obtain
\[
|\varphi_t * (P \chi_{B'})(x)| \le C(n,\varphi, N, \alpha, M) D^{-n}, \quad x\in B_1.
\]
Now, we choose $D$ to be sufficiently large so that the right-hand side of the above inequality is less than $\frac{1}{3}$. This shows that we can find $P$ such that \eqref{lessthanthird} holds, which, together with \eqref{greaterthantwothird}, implies that $\tilde{\chi}$ satisfies \eqref{chitildisbig}. Finally, since $\tilde{\chi}$ satisfies \eqref{Cancellation} and is a bounded function supported in a ball with a radius depending only on $\varphi$, $p$, $\alpha$, and $n$, we conclude
\[
\hprnn{\tilde{\chi}} \le C(\varphi, p, \alpha, n),
\]
which completes the proof.
\end{proof}
We continue by proving the following lemma, which is designed to verify some a priori bounds.
\begin{lemma}\label{appriorilemma}
	Let $0<p,q,r<\infty$ with $\frac{1}{q}=\frac{1}{p}+\frac{1}{r}$, $\varphi$ a Schwartz function with $\int \varphi = 1$, and $u$, a smooth function with compact Fourier support. Also, assume
	\[
	\lqrnn{(\varphi*f)u} \le A \hprnn{f},
	\]
	holds for compactly supported functions $f$. Then $\lrrnn{u} < \infty$.
\end{lemma}

\begin{proof}
After a rescaling, we may assume that $\text{supp}(\hat{u}) \subseteq B_1$. Our strategy is to use the second Plancherel-Polya-Nikolskij inequality \eqref{polyalittlelp} and show that for a sufficiently small choice of $h$, there exists a sequence $\{x_k\}_{k \in \zn}$ with
\begin{equation}\label{polya}
	\|\{u(x_k)\}\|_{l^r(\zn)} < \infty, \quad x_k \in h(k + [0,1]^n), \quad k \in \zn,
\end{equation}
which implies the claim once we apply the Plancherel-Polya-Nikolskij inequality \eqref{polyalittlelp}.
 In order to do this, fix $h$ and partition $\rn$ into cubes of the form
 \[
 Q_k = h(k + [0,1]^n), \quad k \in \zn,
 \]
 then let $\tilde{\chi}$ be the function provided by Lemma \ref{testfunctionlemma} with $\varphi$, $p$, and $\alpha = 2$. Now, for each cube $Q_k$, consider the function
 $f_k = \tau^{hk} \delta^{2\sqrt{n}} \tilde{\chi}$, which, for $h \le 1$, has the property that
 \[
 |\varphi * f_k| \gtrsim \chi_{B_{2\sqrt{n}}(hk)} \ge \chi_{Q_k}, \quad \hprnn{f_k} \lesssim 1, \quad |f_k| \lesssim \chi_{B_c(hk)},
 \]
 where $c$ is a large number provided by the above lemma.
 Next, choose an arbitrary collection of numbers $\{a_k\}_{k \in \zn}$ such that only finitely many of them are nonzero, and let $\{\epsilon_k = \pm 1\}_{k \in \zn}$ be a family of independent random variables with Bernoulli distribution. Then, for the function $f$ defined as
 \[
 f = \sum_{k \in \zn} \epsilon_k a_k f_k,
 \]
 we have
 \begin{equation}\label{khinchin}
 	\lqrnn{(\varphi * f)u}^q = \int \left|\sum_{k \in \zn} \epsilon_k a_k \varphi * f_k\right|^q |u|^q \le A^q \hprnn{f}^q.
 \end{equation}
 Now, for $0<p \le 1$, let us use sub-additivity and estimate the right-hand side as
\[
\hprnn{f} \le \left(\sum_{k \in \zn} |a_k|^p \hprnn{f_k}^p\right)^{\frac{1}{p}} \lesssim \|\{a_k\}_{k \in \zn}\|_{l^p(\zn)},
\]
and for $1 < p < \infty$ we estimate as
\[
\hprnn{f} \lesssim \lprnn{\sum_{k \in \zn} |a_k| \chi_{B_c(hk)}} \lesssim_h \|\{a_k\}_{k \in \zn}\|_{l^p(\zn)}.
\]
Therefore, from \eqref{khinchin} we must have
\[
\int \left|\sum_{k \in \zn} \epsilon_k a_k \varphi * f_k\right|^q |u|^q \lesssim_h A^q \|\{a_k\}_{k \in \zn}\|_{l^p(\zn)}^q.
\]
Now, taking the expectation we obtain
\[
\int \mathbb{E} \left|\sum_{k \in \zn} \epsilon_k a_k \varphi * f_k\right|^q |u|^q = \mathbb{E} \int \left|\sum_{k \in \zn} \epsilon_k a_k \varphi * f_k\right|^q |u|^q \lesssim_h A^q \|\{a_k\}_{k \in \zn}\|^q_{l^p(\zn)},
\]
and then applying Khintchine inequality gives us
\[
\int \left(\sum_{k \in \zn} |a_k \varphi * f_k|^2\right)^{\frac{q}{2}} |u|^q \lesssim_h A^q \|\{a_k\}_{k \in \zn}\|^q_{l^p(\zn)}.
\]
Now, we note that $|\varphi * f_k| \gtrsim \chi_{Q_k}$, and the cubes $Q_k$ are disjoint, which implies that
\[
\left(\sum_{k \in \zn} |a_k|^q \int_{Q_k} |u|^q\right)^{\frac{1}{q}} \lesssim_h A \|\{a_k\}_{k \in \zn}\|_{l^p(\zn)}.
\]
Then, we choose a natural number $N$, and set $\{a_k\}_{k\in\zed^n}$ to be
\[
a_k^p = a_k^q \avr{|u|^q}{Q_k}, \quad |k|\le N, \quad \text{otherwise} a_k=0,
\]
which after plugging into the above inequality, and letting $N$ goes to infinity implies that
\[
\left(\sum_{k\in\zed^n} \left(\avr{|u|^q}{Q_k}\right)^{\frac{r}{q}}\right)^{\frac{1}{r}} \lesssim_h A.
\]
Finally, since $u$ is continuous for any cube $Q_k$, there exists a choice of $x_k$ such that
\[
|u(x_k)| = \avr{|u|^q}{Q_k}^{\frac{1}{q}}, \quad x_k \in \zn,
\]
which shows that \eqref{polya} holds, and this completes the proof.
\end{proof}

In the next lemma, we find a sparse domination of the square function of the symbol of the operator, and as in the dyadic case, it is more convenient to introduce a notation. For a dyadic cube $R$ and $\alpha$, we set
\[
S_\alpha(g|R) := \left(\sum_{2^{-j} \le \alpha l(R)} |\Delta_j^\psi(g)|^2\right)^{\frac{1}{2}} \chi_R.
\]
Also, we use the convention that $2^{-\infty} = 0$.
\begin{lemma}\label{sparsedominationlemma}
	Let $g$ be a tempered distribution, $s > 0$, and $\alpha > \sqrt{n}$. Then, for a dyadic cube $Q_0$ and $0 < \eta < 1$, there exists $\mathcal{C}$, an $\eta$-sparse family of cubes in $Q_0$, with the following properties:\\
	
	For any $Q \in \mathcal{C}$, there exists $\lambda_Q \in \zed \cup \{-\infty\}$ such that
	\begin{align}
		& S_\alpha(g|Q_0)(x) \lesssim \sum_{Q \in \mathcal{C}} 2^{\lambda_Q} \chi_Q(x) + \alpha^{-1} \sum_{Q \in \mathcal{C}} \avr{M_{\nabla\psi}^*(g)^s}{Q}^{\frac{1}{s}} \chi_Q(x), \label{domintaionproperty}\\
		& |Q| \lesssim |\{S_\alpha(g|Q) \ge 2^{\lambda_Q - 1}\}|.  \label{secondmeasureproperty}
	\end{align}
\end{lemma}

\begin{proof}
Using the notation of Theorem \ref{operatorfree}, for any dyadic cube $Q$ and $P \subseteq Q$, let
\[
f_Q = S_\alpha(g|Q), \quad f_{P,Q} = \left(\sum_{\alpha l(P) < 2^{-j} \le \alpha l(Q)} |\Delta^\psi_j(g)|^2 \chi_Q \right)^{\frac{1}{2}}.
\]
First, we note that, by Minkowski's inequality for the $l^2$ norm, the condition $|f_{P,Q}| \le |f_{P}| + |f_{Q}|$ holds. Then, it follows from Theorem \ref{operatorfree} that there exists an $\eta$-sparse family of dyadic cubes $\mathcal{C}$ such that
\[
f_Q(x) \lesssim \sum_{Q \in \mathcal{C}} \gamma_Q \chi_Q(x),
\]
where
\begin{equation}
	\gamma_Q = (f_Q \chi_Q)^*(\eta' |Q|) + (m^{\#}_Q f \chi_Q)^*(\eta' |Q|), \quad \eta' = \frac{1 - \eta}{2^{n+2}}.
\end{equation}
So, it is enough to choose numbers $\lambda_Q$, such that for all $Q\in\mathcal{C}$ \eqref{secondmeasureproperty} holds, and in addition
\[
(f_Q\chi_Q)^*(\eta'|Q|)\lesssim 2^{\lambda_Q}, \quad (m^{\#}_Qf\chi_Q)^*(\eta'|Q|)\lesssim \alpha^{-1}\avr{M_{\nabla\psi}^*(g)(x)^s}{Q}^\frac{1}{s},
\] 
and this proves the claim. First of all, if $(f_Q\chi_Q)^*(\eta'|Q|)=0$, we set $\lambda_Q=-\infty$, and then \eqref{secondmeasureproperty} trivially holds. Now suppose $(f_Q\chi_Q)^*(\eta'|Q|)\neq 0$, and let $\lambda_Q$ be the integer such that
\[
2^{\lambda_Q-1}<(f_Q\chi_Q)^*(\eta'|Q|)\le2^{\lambda_Q},
\]
then from the definition of non-increasing rearrangement we have
\[
\eta'|Q|\le |\{f_Q\chi_Q>2^{\lambda_Q-1}\}|,
\]
which shows that \eqref{secondmeasureproperty}, holds for $Q$ as well. Also, in both case we have
\[
(f_Q\chi_Q)^*(\eta'|Q|)\le 2^{\lambda_Q}.
\]
So, it remains to estimate $(m^{\#}_Qf\chi_Q)^*(\eta'|Q|)$. To this aim, recall that 
\[
m_Q^{\#}f(x)=\sup_{\substack{x\in P\\P\subseteq Q}}\text{osc}(f_{P,Q},P), \quad x\in Q,
\]
and fix $x\in Q$ and $P\subseteq Q$ with $x\in P$. Then for any $y,z\in P$, by the triangle inequality for the $l^2$ norm, we have
\begin{equation}\label{estimateforthediffiernt}
	|f_{P,Q}(y)-f_{P,Q}(z)|\le \left(\sum_{\alpha l(P)<2^{-j}\le \alpha l(Q)}|\Delta^\psi_j(g)(y)-\Delta^\psi_j(g)(z)|^2\right)^{\frac{1}{2}}.
\end{equation}
Now, for a fixed $j$, applying the mean value theorem gives us
\[
|\Delta^\psi_j(g)(y)-\Delta^\psi_j(g)(z)|\lesssim l(P)\sup_{w\in Q}|\nabla\Delta_j(g)(w)|,
\]
which, by recalling that 
\[
\Delta_jg(w)=\psi_{2^{-j}}*g(w), \quad \nabla \psi_{2^{-j}}= 2^j(\nabla\psi)_{2^{-j}},
\]
is equivalent to
\[
|\Delta^\psi_j(g)(y)-\Delta^\psi_j(g)(z)|\lesssim 2^jl(P)\sup_{w\in Q}|(\nabla\psi)_{2^{-j}}(g)(w)|.
\]
Also, since $|w-x|\le \sqrt{n}l(P)\le\sqrt{n}\alpha^{-1}2^{-j}$, it follows from the definition of the non-tangential maximal function that for $\alpha>\sqrt{n}$, we have
\[
|\Delta^\psi_j(g)(y)-\Delta^\psi_j(g)(z)|\lesssim 2^jl(P) M_{\nabla\psi}^*(g)(x).
\]
Now, we estimate \eqref{estimateforthediffiernt} as
\[
|f_{P,Q}(y)-f_{P,Q}(z)|\lesssim M_{\nabla\psi}^*(g)(x)l(P) \sum_{\alpha l(P)<2^{-j}}2^{j}\lesssim \alpha^{-1}M_{\nabla\psi}^*(g)(x),
\]
which proves that
\[
\text{osc}(f_{P,Q},P)\lesssim \alpha^{-1}M_{\nabla\psi}^*(g)(x),
\]
and after taking the supremum over $P\subseteq Q$ gives us
\[
m_Q^{\#}f(x)\lesssim \alpha^{-1}M_{\nabla\psi}^*(g)(x).
\]
Next, we note that for any function $h$ and $0<\lambda<1$, Chebyshev's inequality implies that
\[
h^*(\lambda|Q|)\le \lambda^{-\frac{1}{s}}\avr{|h|^s}{Q}^{\frac{1}{s}},
\]
which, together with the above estimate on $m_Q^{\#}f$, gives us
\[
(m^{\#}_Qf\chi_Q)^*(\eta'|Q|)\lesssim \alpha^{-1}\avr{|M_{\nabla\psi}^*(g)|^s}{Q}^{\frac{1}{s}},
\]
and this completes the proof.

\end{proof}
The last lemma that we need is the following well-known result, which we prove it here.
\begin{lemma}\label{saveragelemma}
	For \(0 < s < r < \infty\), a function \(h\), and an \(\eta\)-sparse collection of cubes \(\mathcal{C}\), we have
	\[
	\left\|\sum_{Q \in \mathcal{C}} \left(\avr{|h|^s}{Q}\right)^{\frac{1}{s}} \chi_Q \right\|_{L^r(\rn)} \lesssim \|h\|_{L^r(\rn)}.
	\]
\end{lemma}

\begin{proof}
Let \( M \) be the cubic Hardy-Littlewood maximal operator. For \( Q \in \mathcal{C} \), let \( E_Q \) be the disjoint parts as in the definition of sparse families. Then we have
\[
\avr{|h|^s}{Q}^{\frac{1}{s}} \le M(|h|^s)^{\frac{1}{s}}(x), \quad x \in Q,
\]
which, after taking the \( r \)-average over \( E_Q \), gives us
\[
\avr{|h|^s}{Q}^{\frac{1}{s}} \le \avr{M(|h|^s)^{\frac{r}{s}}}{E_Q}^{\frac{1}{r}}.
\]
Now we have
\[
\left\|\sum_{Q \in \mathcal{C}} \avr{|h|^s}{Q}^{\frac{1}{s}} \chi_Q \right\|_{L^r(\rn)} \le \left\|\sum_{Q \in \mathcal{C}} \avr{M(|h|^s)^{\frac{r}{s}}}{E_Q}^{\frac{1}{r}} \chi_Q \right\|_{L^r(\rn)},
\]
which, after applying Lemma \ref{lemmasparse} and using the disjointness of the sets \( E_Q \), implies
\[
\left\|\sum_{Q \in \mathcal{C}} \avr{|h|^s}{Q}^{\frac{1}{s}} \chi_Q \right\|_{L^r(\rn)} \lesssim \left(\int M(|h|^s)^{\frac{r}{s}}  \right)^{\frac{1}{r}}.
\]
Finally, using the boundedness of \( M \) on \( L^{\frac{r}{s}}(\rn) \), we obtain
\[
\left\|\sum_{Q \in \mathcal{C}} \left(\avr{|h|^s}{Q}\right)^{\frac{1}{s}} \chi_Q \right\|_{L^r(\rn)} \lesssim \|h\|_{L^r(\rn)},
\]
which completes the proof.

\end{proof}
Now, we proof the third part of Theorem \ref{continuousparp}.
\begin{proof}[Proof of \eqref{hphqclaim}]
Let $A = \|\mathcal{S}_{g,\varphi}\|_{\hprn \to \lqrn}$. As in the proof of \eqref{hphpclaim}, consider the function $g_{i,N}$ and the operator $P_{i,N}$ as defined in \eqref{giNPiN}. For simplicity of notation, we use $g'$ and $P$ instead of $g_{i,N}$ and $P_{i,N}$, respectively, as in \eqref{g'Pf}. Furthermore, we note that, as in the previous case, $\|P\|_{\hprn \to \hqrn} \lesssim A$.
\\

We begin by showing that $g'$ belongs to $\hrrn$. Since it is a finite sum of terms $\Delta^\theta_j(g')$, it is enough to show that for each $j \in \zed$, the function $\Delta^{\theta}_j(g')$ belongs to $\lrrn$. To this aim, we note that the boundedness of the operator $P$ implies that
\[
\lqrnn{ (\varphi_{2^{-j}} * f)\Delta_j^{\theta}(g')} \le A \hprnn{f},
\]
holds for all compactly supported functions. Since $\Delta_j^{\theta}(g')$ has compact Fourier support and $\int \varphi = 1$, an application of Lemma \ref{appriorilemma} implies that $\Delta_j^{\theta}(g')$ belongs to $\lrrn$, and thus $\hrrnn{g'} < \infty$. Next, we fix a dyadic cube $Q_0$ and apply Lemma \ref{sparsedominationlemma} to $g'$ with $s < r$, $\eta = \frac{1}{2}$, and a large $\alpha$ which will be determined later. We then have a sparse collection of cubes $\mathcal{C}$ and numbers $\lambda_Q \in \zed \cup \{-\infty\}$ such that
\begin{align}
	&S_\alpha(g'|Q_0) \lesssim \sum_{Q \in \mathcal{C}} 2^{\lambda_Q} \chi_Q + \alpha^{-1} \sum_{Q \in \mathcal{C}} \avr{M_{\nabla \psi}^*(g')^s}{Q}^\frac{1}{s} \chi_Q \label{g'domintaionproperty},\\
	&|Q| \lesssim |\{S_\alpha(g'|Q) \ge 2^{\lambda_Q - 1}\}| \label{g'secondmeasureproperty}.
\end{align}
Now, \eqref{g'domintaionproperty} implies that
\begin{equation}\label{sg'lrnormfirstestimate}
	\lrrnn{S_\alpha(g'|Q_0)} \lesssim \lrrnn{\sum_{Q \in \mathcal{C}} 2^{\lambda_Q} \chi_Q(x)} + \alpha^{-1} \lrrnn{\sum_{Q \in \mathcal{C}} \avr{M_{\nabla \psi}^*(g')(x)^s}{Q}^\frac{1}{s} \chi_Q}.
\end{equation}
For the first term, we use Lemma \ref{lemmasparse} and get
\begin{equation}
	\lrrnn{\sum_{Q \in \mathcal{C}} 2^{\lambda_Q} \chi_Q(x)} \lesssim \left(\sum_{Q \in \mathcal{C}} 2^{r \lambda_Q} |Q|\right)^{\frac{1}{r}},
\end{equation}
and an application of Lemma \ref{saveragelemma} provides the following estimate for the second term:
\[
\lrrnn{\sum_{Q \in \mathcal{C}} \avr{M_{\nabla \psi}^*(g')(x)^s}{Q}^\frac{1}{s} \chi_Q} \lesssim \lrrnn{M_{\nabla \psi}^*(g')} \lesssim \hrrnn{g'},
\]
where, in the last line, we used the boundedness of non-tangential maximal functions on $\hrrn$. Putting the above two bounds together with \eqref{sg'lrnormfirstestimate}, we obtain
\begin{equation}\label{salphanorm}
	\lrrnn{S_\alpha(g'|Q_0)} \lesssim \left(\sum_{Q \in \mathcal{C}} 2^{r \lambda_Q} |Q|\right)^{\frac{1}{r}} + \alpha^{-1} \hrrnn{g'}.
\end{equation}
Now, we proceed to estimate the main term of the above inequality. To do this, let $\tilde{\chi}$ be the function provided by Lemma \ref{testfunctionlemma}. Then, for each cube $Q\in\mathcal{C}$, with center $c_Q$ let 
\[
\tilde{\chi}_Q(x)= \tilde{\chi}(\frac{x-c_Q}{2\sqrt{n}l(Q)})=\tau^{c_Q}\delta^{2\sqrt{n}l(Q)}\tilde{\chi}(x), \quad x\in\rn.
\]
Here, we summarize properties of the above function:	
\begin{align}
	&|\varphi_t*\tilde{\chi}_Q|\ge \frac{1}{3}\chi_Q, \quad 0<t\le\alpha l(Q),\label{thefirstpropertyoftild}\\
	&|\tilde{\chi}_Q|\lesssim \chi_{cQ},\label{thescondpropertioftild}\\
	&\hprnn{\tilde{\chi}_Q}\lesssim |Q|^{\frac{1}{p}}\label{thethirdpropertyoftild}.
\end{align}	
	To see the first property note that from Lemma \ref{testfunctionlemma}
	we have
	\[
	|\varphi_t*\tilde{\chi}_Q|=|\tau^{c_Q}\delta^{2\sqrt{n}l(Q)}(\varphi_{\frac{t}{2\sqrt{n}l(Q)}}*\tilde{\chi})|\ge\frac{1}{3}\chi_{B_{2\sqrt{n}l(Q)}(c_Q)}\ge\frac{1}{3}\chi_Q, \quad \frac{t}{2\sqrt{n}l(Q)}\le\alpha.
	\]
The second and third properties also follow from the properties of $\tilde{\chi}$ and from using dilation and translation. From now on, we fix a finite sub-collection of $\mathcal{C}$ and denote it by $\mathcal{C}'$.
 Then, take a sequence of independent Bernoulli random variables $\{\epsilon_Q = \pm 1\}_{Q \in \mathcal{C}'}$ and consider the following random function
 
\[
f_{\epsilon}=\sum_{Q\in\mathcal{C}'}\epsilon_Q2^{t\lambda_Q}\tilde{\chi}_Q, \quad t=\frac{r}{p}.
\]
The first thing to note is that
\begin{equation}\label{hpnormfepsilon}
	\hprnn{f_\epsilon}\lesssim \left(\sum_{Q\in\mathcal{C}'}2^{r\lambda_Q}|Q|\right)^{\frac{1}{p}}, \quad 0<p<\infty.
\end{equation}
	To see this, we note that for $0<p\le1$ sub-additivity implies
	\[
		\hprnn{f_\epsilon}\le \left(\sum_{Q\in\mathcal{C}'}2^{tp\lambda_Q}\hprnn{\tilde{\chi}_Q}^p\right)^{\frac{1}{p}}\lesssim \left(\sum_{Q\in\mathcal{C}'}2^{r\lambda_Q}|Q|\right)^{\frac{1}{p}},
	\]
where the last estimate follows from \eqref{thethirdpropertyoftild}. Also, for $1<p<\infty$, from \eqref{thescondpropertioftild} we have 
\[
\hprnn{f_\epsilon}\lesssim\lprnn{f_{\epsilon}}\lesssim\lprnn{\sum_{Q\in\mathcal{C}'}2^{t\lambda_Q}\chi_{cQ}},
\]
which after noting that the collection of concentric dilations $\{cQ: Q\in\mathcal{C'}\}$ is $c^{-n}\eta$-sparse, and using Lemma \ref{lemmasparse} implies that
\[
\hprnn{f_\epsilon}\lesssim \lprnn{\sum_{Q\in\mathcal{C}'}2^{t\lambda_Q}\chi_{cQ}}\lesssim \left(\sum_{Q\in\mathcal{C}'}2^{r\lambda_Q}|Q|\right)^{\frac{1}{p}},
\]	
which proves the claim.	Next, recall that in the defining expression of the operator $P$
\[
P(f)=\sum_{j\in\zed}\varphi_{2^{-j}}*f\Delta_j^\theta(g')=\sum_{\substack{j\in m\zed+i\\|j|\le N}}\varphi_{2^{-j}}*f\Delta_j^\psi(g),
\]
every two consecutive terms have Fourier support at sufficiently large distance from each other provided by the large magnitude of $m$. Therefore, if we take a sequence of independent Bernoulli random variables $\{\omega_j=\pm1\}_{j\in\zed}$ and modify the operator $P$ as
\[
P_\omega(h):=\sum_{\substack{j\in m\zed+i\\|j|\le N}}\omega_j\varphi_{2^{-j}}*h\Delta_j^\psi(g),
\]	
we still have
\[
S_\theta(P_\omega(h))(x)=\left(\sum_{\substack{j\in m\zed+i\\|j|\le N}}|\varphi_{2^{-j}}*h\Delta_j^\psi(g)|(x)^2\right)^\frac{1}{2}\le\mathcal{S}_{g,\varphi}(h)(x), \quad x\in\rn,
\]	
which implies that
\begin{equation}
	\|P_\omega\|_{\hprn\to\hqrn}\simeq\|P_\omega\|_{\hprn\to\dot{H}^q(\rn)}\lesssim A.
\end{equation}	
Now, for a fixed $\omega$ and $\epsilon$ we have
\[
P_\omega(f_\epsilon)=\sum_{j\in\zed}\Delta_j(g') \omega_j\varphi_{2^{-j}}*\sum_{Q\in\mathcal{C}'}\epsilon_Q2^{t\lambda_Q}\tilde{\chi}_Q=\sum_{Q\in\mathcal{C}'}\epsilon_Q2^{t\lambda_Q}\sum_{j\in\zed} \omega_j\varphi_{2^{-j}}*\tilde{\chi}_Q\Delta_j(g').
\]
Then, we get
\[
\lqrnn{P_\omega(f_\epsilon)}\le \hqrnn{P_\omega(f_\epsilon)}\le A\hprnn{f_\epsilon},
\]	
and from \eqref{hpnormfepsilon} we obtain
\[
\lqrnn{P_\omega(f_\epsilon)}\lesssim A \left(\sum_{Q\in\mathcal{C}'}2^{r\lambda_Q}|Q|\right)^{\frac{1}{p}},
\]	
or equivalently

\[
\int_{\rn} \left|\sum_{Q\in\mathcal{C}'}\epsilon_Q2^{t\lambda_Q}\sum_{j\in\zed} \omega_j\varphi_{2^{-j}}*\tilde{\chi}_Q(x)\Delta_j(g')(x)\right|^qdx \lesssim A^q \left(\sum_{Q\in\mathcal{C}'}2^{r\lambda_Q}|Q|\right)^{\frac{q}{p}}.
\]	
Taking expectation with respect to $\epsilon$ first, and using Khintchine inequality gives us
\[
\int_{\rn} \left(\sum_{Q\in\mathcal{C}'}2^{2t\lambda_Q}\left|\sum_{j\in\zed} \omega_j\varphi_{2^{-j}}*\tilde{\chi}_Q(x)\Delta_j(g')(x)\right|^2\right)^\frac{q}{2}dx \lesssim A^q \left(\sum_{Q\in\mathcal{C}'}2^{r\lambda_Q}|Q|\right)^{\frac{q}{p}},
\]	
and then taking expectation with respect to $\omega$ implies that
\[
\int_{\rn}\mathbb{E} \left(\sum_{Q\in\mathcal{C}'}2^{2t\lambda_Q}\left|\sum_{j\in\zed} \omega_j\varphi_{2^{-j}}*\tilde{\chi}_Q(x)\Delta_j(g')(x)\right|^2\right)^\frac{q}{2}dx \lesssim A^q \left(\sum_{Q\in\mathcal{C}'}2^{r\lambda_Q}|Q|\right)^{\frac{q}{p}}.
\]	
Now, let us call
\begin{equation}
	F(x)=\mathbb{E} \left(\sum_{Q\in\mathcal{C}'}2^{2t\lambda_Q}\left|\sum_{j\in\zed} \omega_j\varphi_{2^{-j}}*\tilde{\chi}_Q(x)\Delta_j(g')(x)\right|^2\right)^\frac{q}{2}, \quad x\in\rn,
\end{equation}	
then the above inequality is nothing but
\begin{equation}\label{lonenormf}
	\lonernn{F}\lesssim A^q \left(\sum_{Q\in\mathcal{C}'}2^{r\lambda_Q}|Q|\right)^{\frac{q}{p}}.
\end{equation}	
Our next task is to show that 
\begin{equation}\label{measuredensityoff}
	|Q|\lesssim |\{F\gtrsim 2^{r\lambda_Q}\}\cap Q|, \quad Q\in\mathcal{C'}.
\end{equation}	
In order to do this, we note that
\[
F(x)\ge 2^{tq\lambda_Q}\mathbb{E}\left|\sum_{j\in\zed} \omega_j\varphi_{2^{-j}}*\tilde{\chi}_Q(x)\Delta_j(g')(x)\right|^q,
\]	
which after using Khintchine inequality implies that
\[
F(x)\gtrsim 2^{tq\lambda_Q} \left(\sum_{j\in\zed} |\varphi_{2^{-j}}*\tilde{\chi}_Q(x)\Delta_j(g')(x)|^2\right)^\frac{q}{2}, \quad x\in \rn.
\]	
Also, from \eqref{thefirstpropertyoftild}
\[
|\varphi_{2^{-j}}*\tilde{\chi}_Q|\ge \frac{1}{3}\chi_Q, \quad 2^{-j}\le\alpha l(Q),
\]
we obtain
\[
F(x)\gtrsim 2^{tq\lambda_Q} \left(\sum_{2^{-j}\le\alpha l(Q)} |\Delta_j(g')(x)|^2\right)^\frac{q}{2} =2^{tq\lambda_Q} S_\alpha(g'|Q)(x)^q, \quad x\in Q.
\]
Now, recall \eqref{g'secondmeasureproperty} stating that
\[
|Q|\lesssim |\{S_\alpha(g'|Q)(x)\ge 2^{\lambda_Q-1}\}|,
\]
which together with the above inequality implies that
\[
|Q|\lesssim |\{F(x)\gtrsim 2^{q(t+1)\lambda_Q}\}\cap Q|.
\]	
Now, it is enough to note that 
\[
q(t+1)=r, \quad t=\frac{r}{p}, \quad \frac{1}{q}=\frac{1}{p}+\frac{1}{r},
\]	
which proves \eqref{measuredensityoff}. Having this inequality in hand, we can follow the same line of reasoning as in the dyadic case, which we will do now. First, let us partition cubes in $\mathcal{C'}$ as
\[
\mathcal{C'}_k=\{Q\in\mathcal{C'}: \lambda_Q=k\},\quad k\in\zed\cup\{-\infty\}.
\]
Then we have
\[
\sum_{Q\in\mathcal{C}'}2^{r\lambda_Q}|Q|=\sum_{k\in\zed}2^{kr}\sum_{Q\in\mathcal{C}'_k}|Q|.
\]		
So, if $\mathcal{C''}_k$ is the collection of maximal cubes in $\mathcal{C'}_k$, it follows from sparseness of $\mathcal{C'}_k$ that we have
\[
\sum_{k\in\zed}2^{kr}\sum_{Q\in\mathcal{C}'_k}|Q|\lesssim \sum_{k\in\zed}2^{kr}\sum_{Q\in\mathcal{C}''_k}|Q|,
\]	
which together with \eqref{measuredensityoff}, and noting that maximal cubes in $\mathcal{C''}_k$ are disjoint implies that
\[
\sum_{Q\in\mathcal{C}'}2^{r\lambda_Q}|Q|\lesssim \sum_{k\in\zed}2^{kr}\sum_{Q\in\mathcal{C}''_k}|\{F\gtrsim 2^{kr}\}\cap Q|\lesssim \sum_{k\in\zed}2^{kr}|\{F\gtrsim 2^{kr}\}|\lesssim \lonernn{F}.
\]	
Next, we use \eqref{lonenormf} and obtain
\[
\sum_{Q\in\mathcal{C}'}2^{r\lambda_Q}|Q|\lesssim A^q\left(\sum_{Q\in\mathcal{C}'}2^{r\lambda_Q}|Q|\right)^{\frac{q}{p}},
\]	
and noting that the right hand side is finite we get
\[
\left(\sum_{Q\in\mathcal{C}'}2^{r\lambda_Q}|Q|\right)^\frac{1}{r}\lesssim A,
\]	
and after taking the supremum over all finite sub-collections of $\mathcal{C}$ we get
\[
\left(\sum_{Q\in\mathcal{C}}2^{r\lambda_Q}|Q|\right)^\frac{1}{r}\lesssim A.
\]	
Then, we recall \eqref{salphanorm} and obtain
\[
\lrrnn{S_\alpha(g'|Q_0)}\lesssim A+\alpha^{-1}\hrrnn{g'}.
\]		
Now, we consider $2^n$ large dyadic cubes of the form $[0,\pm2^L]^n$, apply the above inequality to each of them and then after letting $L$ tends to infinity we finally get
\[
\lrrnn{S(g')}\lesssim A+\alpha^{-1}\hrrnn{g'}.
\]	
So, there exists a polynomial $U$ such that
\[
\hrrnn{g'-U}\lesssim \lrrnn{S(g')} \lesssim A+\alpha^{-1}\hrrnn{g'},
\]
however since we already showed that $g'\in\lrrn$, the polynomial $U$ must be zero and we conclude
\[
\hrrnn{g'}\lesssim A+\alpha^{-1}\hrrnn{g'}.
\]	
Now by choosing $\alpha$ large enough and noting that $\hrrnn{g'}$ is finite, we get $\hrrnn{g'}\lesssim A$. Then, recall that $g'=g_{i,N}$ and 
\[
g_N=\sum_{|j|\le N}\Delta_jg =\sum_{0\le i<m} g_{i,N},
\]
which implies that the sequence of functions
\[
S_N(g)=(\sum_{|j|\le N}|\Delta_jg|^2)^{\frac{1}{2}},
\]
is bounded in $\lrrn$, and thus Fatou lemma implies
\[
\lrrnn{S(g)}\lesssim A,
\]	
which means that there exists a polynomial $U'$ such that
\[
\hrrnn{g-U'}\lesssim A,
\] 	
and this prove \eqref{hphqclaim}, and completes the proof of Theorem \ref{continuousparp}.	
\end{proof}
\section{Acknowledgments}
I would like to thank the referees for valuable suggestions and comments on the manuscript.
\section{Declarations}
\subsection{Funding} Not applicable
\subsection{Conflict of interest}
On behalf of all authors, the corresponding author states that there is no conflict of interest. (Not applicable)
\subsection{Ethics approval and consent to participate} Not applicable
\subsection{Consent for publication} Not applicable
\subsection{Data availability} Not applicable
\subsection{Materials availability} Not applicable
\subsection{Code availability} Not applicable 


\end{document}